\documentclass[12pt, a4paper]{article}
\usepackage{amsmath,amssymb,bm,graphicx,latexsym,graphics,epsfig,subfigure,color}
\usepackage{amsthm}

\usepackage{mathrsfs}
\usepackage{amssymb,amsmath,amsfonts,longtable,color}
\usepackage{amsbsy, amsmath, amssymb, multirow, epstopdf,epsf}
\usepackage{setspace}
\usepackage{amsthm}
\usepackage{graphicx,float,wasysym}
\usepackage{color}
\usepackage{enumerate} 
\usepackage{authblk} 

\usepackage{graphicx}			
\usepackage{mathrsfs}
\usepackage{amsfonts}
\usepackage{amsmath,amssymb,mathrsfs,bm,latexsym,graphicx,color}		
\usepackage{amssymb}
\usepackage[dvipsnames]{xcolor}
\usepackage{natbib} \bibpunct{(}{)}{;}{a}{,}{,} 
\usepackage[normalem]{ulem}
\usepackage[colorlinks=true,urlcolor=blue,citecolor=blue,linkcolor=blue,bookmarks=true]{hyperref}
\usepackage{tikz}
\usepackage{circuitikz}
\usepackage{multicol}
\usepackage{multirow}

\usepackage[T1]{fontenc}
\usepackage{booktabs}

\newcommand{\Frac}[2]{\frac{\displaystyle#1}{\displaystyle#2}}

\allowdisplaybreaks

\newtheorem{definition}{Definition}
\newtheorem{theorem}{Theorem}

\newtheorem{corollary}{Corollary}
\newtheorem{lemma}{Lemma}
\newtheorem{example}{Example}
\newtheorem{remark}{Remark}

\usepackage{tikz,xcolor,hyperref}

\usepackage{marvosym}

\title{\bf Orderings of the finite mixture with modified proportional hazard rate model}
\author{
	Lina Guo \thanks{The corresponding author. E-mail: 2022212115@nwnu.edu.cn}\hspace{-1.5mm} ~~
	Rongfang Yan  \\
	{College of Mathematics and Statistics }\\
	{Northwest Normal University, Lanzhou 730070, China}\\ \medskip   \medskip   
	}

\begin{document}
	\maketitle
\begin{abstract}
		
	In this paper, we consider finite mixture models with modified proportional hazard rates. Sufficient conditions for the  usual stochastic order and the hazard order are established under chain majorization. We study stochastic comparisons under different settings of T-transform for various values of chain majorization. We establish usual stochastic order and hazard rate order between two mixture random variables when a matrix of model parameters and mixing proportions changes to another matrix in some mathematical sense. Sufficient conditions for the star order and Lorenz order are established under weakly supermajorization. The results of this paper are illustrated with numerical examples.
		
		\medskip
		
		
		
\end{abstract}
	\newpage
\section{Introduction}
	Finite mixture, as a very effective tool for modeling heterogeneity, have found applications in many areas of reliability theory, survival analysis. In many applications, it is often the case that the data arise from a mixture population of two or more heterogeneous subpopulations, see \cite{amini2017stochastic}. In other words, the finite mixture models represents datasets that are assumed to possess more than one distribution combined in different proportions. For example, the manufactured engineering items are often heterogeneous due to different reasons such as the quality of resources and components used in the production process, the quality of raw materials, human errors, different work shifts, environmental conditions, and so on, produced items have often different lifetime distributions. When the items are mixed they lead to a heterogeneous population, see \cite{finkelstein2008failure} and \cite{cha2013failure}. In this regard, the finite mixture model can be used for modeling lifetime date arising from a finite amount of heterogeneous subpopulations.

	Let $\textbf{X}=(X_{1},X_{2},\cdots,X_{n})$ be a vector of $n$ random variables with the cumulative distribution function, survival function, probability density function and hazard rate function are defined as $F_{X}(\cdot),\bar{F}_{X}(\cdot)$ $f_{X}(\cdot)$ and $r_{X}(\cdot)$, respectively. Suppose that we have $n$ homogeneous infinity subgroups of units with $X_{i}$ denoting the lifetime of a unit in the $i$th subpopulation, $i=1,2,\cdots,n$. Then, let $C(X,p)$ be a random variable representing the mixture of units that are drawn from these $n$ subpopulations.
	Evidently, its distribution, survival and density functions can be given by
    
    $$F_{C_{n}(X,p)}(x)=\sum_{i=1}^{n}p_{i}F_{X_{i}}(x),$$

   $$\bar{F}_{C_{n}(X,p)}(x)=\sum_{i=1}^{n}p_{i}\bar{F}_{X_{i}}(x),$$

   $$f_{C_{n}(X,p)}(x)=\sum_{i=1}^{n}p_{i}f_{X_{i}}(x),$$
    respectively, where $p_{i}(>0)$ is the mixing proportion such that $\sum_{i=1}^{n}p_{i}=1$.
	Since the hazard rate is a conditional feature, the corresponding mixed hazard rate is defined by modifying the conditional weight, see \cite{navarro2004obtain}, \cite{cha2013failure}, \cite{finkelstein2008failure}.
	\cite{navarro2017stochastic} discussed that stochastic comparisons of two mixtures with distorted distributions in coherent systems. 
	\cite{hazra2018stochastic} discussed that stochastic comparisons of two finite mixtures where corresponding random variables follow proportional hazard rate(PHR), proportional reversed hazards(PHR) or accelerated lifetime model, using the concept of multivariate chain majorization order.
	\cite{bansal2019stochastic} discussed a particular case of the multivariate mixture model in which the baseline distribution function is represented in terms of a copula and study stochastic comparisons among the two random vectors.
	\cite{albabtain2020stochastic} showed that stochastic comparisons of finite mixture models under  Weibull distribution was carried out by adjusting different weight functions.
	\cite{sattari2021stochastic} showed usual stochastic order, hazard rate order and likelihood ratio order of finite mixture models with generalized Lehmann distributed components thought chain majorization. It discussed stochastic comparisons under various parameter conditions.
	\cite{barmalzan2022dispersion} showed that star, convex transform orders, dispersive, right spread and mean residual life orders of mixture exponential distributions by majorization orders.
	\cite{barmalzan2022orderings} discussed that usual stochastic order, hazard rate order and reversed hazard rate order for two finite mixtures with location-scale family distributed components under chain majorization order. 
	\cite{panja2022stochastic} derived some stochastic comparison results for two finite mixture models where corresponding random variables follow proportional odds(PO), PHR or PRH model. 
	\cite{nadeb2022new} derived usual stochastic order, hazard rate order, reverse hazard rate order, likelihood ratio order and dispersive order of the finite mixture in the sense of vector majorization. 
	\cite{shojaee2023some} discussed that stochastic comparisons of arithmetic and geometric mixture models in the sense of usual stochastic order, upper orthant order, weak hazard rate order and likelihood ratio order.
	\cite{bhakta2023stochastic} discussed that stochastic comparisons of two finite mixtures with inverted-Kumaraswamy distribution by majorization conditions. 
	\cite{fang2023usual} obtained the first sample maxima is larger than the second sample maxima with respect to the usual stochastic order by chain majorization. 
	\cite{bhakta2023stochastic} showed stochastic comparisons of the finite mixtures of general family of distributions by majorization, P-larger majorization, reciprocally and chain majorization.
	Some scholars studied $\alpha$-mixture models.
	\cite{asadi2019alpha} derived the constant elasticity of substitution of $\alpha$-mixture.
	\cite{shojaee2021some} derived some ageing properties, additive and multiplicative and bending properties of $\alpha$-mixtures. It discussed usual stochastic order and hazard rate order of finite $\alpha$-mixture and continuous $\alpha$-mixture.
	\cite{barmalzan2021stochastic} discussed usual stochastic order, hazard rate order, reversed hazard rate order and dispersive order of the finite $\alpha$-mixture model under majorization.
	\cite{shojaee2022stochastic} showed usual stochastic order, hazard rate and aging properties of the generalized finite $\alpha$-mixture. At the same time, it described conditional characteristics, additive and multiplicative models of the generalized finite $\alpha$-mixture.
	\cite{bhakta2024ordering} discussed usual stochastic order and reversed hazard rate order between two multiple-outlier finite $\alpha$-mixtures.

   	\cite{balakrishnan2018modified} defined a new distribution, which is called modified proportional hazard rates (MPHR) model. The MPHR model is very important model in reliability theory and survival analysis. A random variable $X$ is said to follow the MPHR model with tilt parameter $\alpha$, modified proportional hazard rate $\lambda$ and baseline survival function $\bar{F}$(denoted by $MPHR(\alpha;\lambda;\bar{F})$) if and only if,
    \begin{equation}
    	\bar{F}(x;\alpha,\lambda)=\frac{\alpha\bar{F}^{\lambda}(x)} {1-\bar{\alpha}\bar{F}^{\lambda}(x)},\quad 
    	x,\lambda,\alpha\in\mathcal{R}^+,\quad
    	\bar{\alpha}=1-\alpha.
    	\label{1}
    \end{equation}

    The finite mixture with modified proportional hazard rate in (\ref{1}) contains the following serval models as special cases.
    	 If $\alpha=1$, (\ref{1}) reduces to the PHR model.
    	 If $\lambda=1$, (\ref{1}) reduces to the PO model.

    The MPHR model in includes some well-known distributions such as the extended exponential and extended Weibull distributions(\cite{marshall2007life}), extended Pareto distribution(\cite{ghitany2005marshall}) and extended Lomax distribution(\cite{ghitany2007marshall}). 
    \cite{navarro2017stochastic} showed the  origin of MPHR, and discussed the relative hazard rate order by exploring the range of different values for $\alpha$.
    \cite{yan2023stochastic} Studied the usual stochastic order and hazard rate order of the second-order statistics from MPHR under different parameter conditions through majorization order.  
   
    Through the study of the literature, We extend this study to stochastic comparisons of mixture models under the MPHR model. We discuss the usual stochastic order or the hazard order by chain majorization, and star order or Lorenz order by weakly supermajorization of the finite mixture with MPHR model. We analyze two scenarios when discussing the usual stochastic order, one scenario involves mixing proportions and tilt parameters being different, while the other involves mixing proportions and modified proportional hazard rates being different. Additionally, we provided specific numerical examples and used line graphs to validate our research findings.

    In this paper, we discuss some stochastic comparisons of two finite mixture models with the components having MPHR distributions. Section \ref{sec-2} reviews some basic concepts that will be used in the sequel. Section \ref{sec-3} discusses the usual stochastic order of two finite mixture models with MPHR distributed components. Section \ref{sec-4} discusses the hazard rate order of two finite mixture models with MPHR distributed components.
    Section \ref{sec-5} discuss the star order and Lorenz order. Several examples are also presented to illustrate all
    the established results. Finally, some concluding remarks are made in Section \ref{sec-6}.

\section{Preliminaries}\label{sec-2}  
     In this section,  we present some basic definitions and useful lemmas that are essential for subsequent developments.
     
     Let X and Y are two non-negative continuous random variables with density functions $f_{X}(t)$, $f_{Y}(t)$,
     distribution functions $F_{X}(t)$ and $F_{Y}(t)$,
     survival functions $\bar{F}_{X}(t)=1-F_{X}(t)$, $\bar{F}_{Y}(t)=1-F_{Y}(t)$, 
     and hazard rate functions $r_{X}(t)=f_{X}(t)/\bar{F}_{X}(t)$, $r_{Y}(t)=f_{Y}(t)/\bar{F}_{Y}(t)$.
     Let $F^{-1}_{X}(t)$ and $F^{-1}_{Y}(t)$
     be the right continuous inverses of 
     $F_{X}(t)$ and $F_{Y}(t)$, respectively. 
     In addition, we use $a\overset{\rm sgn}=b$ to denote that both sides of the equality have the same sign.
\begin{definition}\label{def-1}
	A random variable $X$ is said to be smaller than $Y$ in the 
	\begin{enumerate} [\rm (i)]
		\item  usual stochastic order (denoted by $X \le_{\rm st} Y$) if  $\overline F_X(t) \le \overline F_Y(t)$, for all $t \in \mathbb R;$
		\item  { hazard rate order} (denoted by $X\leq_{\rm hr}Y$) if
		$r_{X}(t)\ge r_{Y}(t)$ for all $t\in\mathbb{R}$, or  equivalently, if $\overline {F}_Y(t)/\overline {F}_X(t)$ is increasing in $t\in\mathbb{R};$
		\item  { star order} (denoted by $X\leq_{\rm \star}Y$) if
		$F^{-1}_{Y}[F_{X}(x)]/x$ is increasing in $x\in\Bbb{R}_{+};$
		\item Lorenz order (denoted by $X\le_{lorenz} Y$) if $L_{X}(p)\ge L_{Y}(p)$ for all $p\in[0,1]$, where the Lorenz curve $L_{X}$, corresponding to $X$ is defined as
		 \begin{center}
		 	$L_{X}(p)=\frac{\int_{0}^{p}F^{-1}(u)\mathrm{d}u}
		 	{\int_{0}^{1}F^{-1}(u)\mathrm{d}u}$
		 \end{center}
	\end{enumerate}	
\end{definition}
     It is to be noted that $X\le_{\rm hr} Y\Rightarrow X\le_{\rm st} Y$ and $X\le_{\star} Y\Rightarrow X\le_{Lorenz} Y$, see \cite{shaked2007stochastic}.
\begin{definition}\label{def-2}
    Let $a_{1}\le a_{2}\le \cdots \le a_{n}$ and $b_{1} \le b_{2}\le \cdots \le b_{n}$ be the increasing arrangements of
    $\textbf{a}=(a_{1},a_{2},\cdots ,a_{n})$ and $\textbf{b}=(b_{1},b_{2},\cdots ,b_{n})$ , 
    respectively. 
    \begin{enumerate}[\rm(i)]
    \item $\textbf{a}$ is said to majorize $\textbf{b}$, 
    denoted by $\textbf{a}\overset{m}\succeq \textbf{b}$,
    if $\sum_{i=1}^{j} a_{(i)}\le\sum_{i=1}^{j}b_{(i)}$,
    for $i=1,\cdots , n-1$, 
    and $\sum_{i=1}^{n}a_{(i)}=\sum_{i=1}^{n} b_{(i)};$
    \item $\textbf{a}$ is said to weakly supermajorize $\textbf{b}$,
    denoted by $\textbf{a}\overset{w}\succeq \textbf{b}$,
    if $\sum_{i=1}^{j} a_{(i)}\le\sum_{i=1}^{j}b_{(i)}$,
    for all $i=1,2,\cdots,n$.
    \end{enumerate}
    
    It is to be noted that 
    $\textbf{a}\overset{w}\succeq \textbf{b}
    \Rightarrow
    \textbf{a}\overset{m}\succeq \textbf{b}$,
    see \cite{marshall1979inequalities}.
\end{definition}

$\\$
Any T-transform matrix has the form
   \begin{center} 
	$T=\omega I_{n}=(1-\omega)\Pi_{n}$
   \end{center}

   Where $0\le \omega \le 1$, $\Pi_{n}$ is a permutation matrix. $I_{n}$ is a identity matrix.
   We say that two T-transform matrices $T_{1}=\omega_{1}I_{n}+(1-\omega_{1})\Pi_{1}$ and $T_{2}=\omega_{2}I_{n}+(1-\omega_{2})\Pi_{2}$ have the same structure if $\Pi_{1}=\Pi_{2}$. It is well-known that a finite product of T-transform matrices with the same structure is also a T-transform matrix, while this product may not be a T-transform matrix if its elements do not have the same structure, see \cite{balakrishnan2014stochastic}.\\

\begin{definition}\label{def-3}
	Consider the $m\times n$ matrices $A=\left \{a_{ij}\right \} $ and $B=\left \{b_{ij}\right\}$ with rows $a_{1},\cdots,a_{m}$ and $b_{1},\cdots,b_{m}$, respectively.
	\begin{enumerate}[\rm(i)]
	\item If ${\textbf{a}_{i}}\overset{\rm{m}}{\succeq}{\textbf{b}_{i}}$, for $i=1,\cdots,m$,
	then we say that matrix A is larger than the matrix B in row majorization and is denoted by ${A} >^{row}{B};$
	\item If there exists a finite set of $n \times n$ T-transform matrices $T_{1},\cdots,T_{k}$ such that $B=AT_{1}T_{2}\cdots T_{k}$,
    then we say that the matrix A is larger than matrix B in chain majorization and is denoted by $A\gg B$.
	\end{enumerate}
	
  It is to be noted that $A\gg B\Longrightarrow {A}>^{row} {B}$, see \cite{marshall1979inequalities} for an elaborate discussion on the theory of vector and
  matrix majorizations and their applications.
\end{definition}
\begin{definition}\label{def-4}
	$\left[\text{\cite{marshall1979inequalities}}\right]$ A real-valued function G defined on a set $\mathbb{A}\subseteq \mathbb{R}$ is said to be Schur-convex (Schur-concave) on $\mathbb{A}$, if
	\begin{center}
		$\textbf{a}\overset{m}\succeq\textbf{b} \Longrightarrow G(a)\ge G(b)$ \quad
		for any $a,b\in\mathbb{A}$.
	\end{center}
\end{definition}
\begin{lemma}\label{lem-1}
	$\left[\text{\cite{barmalzan2022orderings}}\right]$
	Consider a differentiable function G: ${R_{+}}^{4}\to R_{+}$

		\begin{center}
				$G(A)\ge(\le)G(B)$ for all A, B such that $A\in \mathcal K_{2}(\mathcal L_{2})$, and $A\gg B$ 
		\end{center}          
	
	if and only if
	\begin{enumerate}[\rm(i)]
		\item $G(A)=G(A \Pi)$ for all permutation matrices $\Pi$, and for all $A\in 
		\mathcal K_{2}(\mathcal L_{2});$
		\item $\sum_{i=1}^{2}(a_{ik}-a_{ij})\left(G_{ik}(A)-G_{ij}(B)\right) \ge (\le)0$ for all $j,k=1,2$, and for all $A\in \mathcal K_{2}(\mathcal L_{2})$, where $G_{ij}(A)=\partial G(A)/ \partial a_{ij}$.
	\end{enumerate}
\end{lemma}       
\begin{lemma}\label{lem-2}
	$\left[\text{\cite{barmalzan2022orderings}}\right]$
	 Consider a differentiable function H: ${R_{+}}^{2}\to R_{+}$, and let the function $G_{n}: {R_{+}}^{2n}\to R_{+}$ be defined as $G_{n}(A)=\sum_{i=1}^{n} H(a_{1i},a_{2i})$. If $G_{2}$ satisfies Lemma \ref{lem-1}, then $G_{n}(A)\ge G_{n}(B)$, where $ A \in \mathcal K_{n}(L_{n})$ and $B=AT$.
\end{lemma}

\begin{lemma}\label{lem-3}
	$\left[\text{\cite{saunders1978quantiles}}\right]$ 
	Let $\left\{F_{a},a\in\Bbb R_{+}\right\}$ be a class of distribution function, such that $F_{a}$ is supported on some interval $(c,d)\subseteq (0,\infty)$ and has a density $f_{a}$ which does not vanish on any subinterval of $(c,d)$. Then $F_{a}\ge_{\star} F_{b}$ for $a\ge b$, if and only if $\frac{\partial F_{a}(x)/\partial a}{xf_{a}(x)}$ is decreasing in $x$.
\end{lemma}
 
   $\\$
   Also, we define two required spaces as the below  \\  
     
   $\mathcal{K}_{n}=
   \left \{ (a,b)=\begin{pmatrix}
   	a_{1}&\cdots&a_{n} \\
   	b_{1}&\cdots&b_{n}
   \end{pmatrix}:a_{i},b_{j}>0 \quad and \quad (a_{i}-a_{j})(b_{i}-b_{j})\le 0,i,j=1,\cdots,n \right \}  $,\\
   
   $\mathcal{L}_{n}=
   \left \{ (a,b)=\begin{pmatrix}
   	a_{1}&\cdots&a_{n} \\
   	b_{1}&\cdots&b_{n}
   \end{pmatrix}:a_{i},b_{j}>0 \quad and \quad (a_{i}-a_{j})(b_{i}-b_{j})\ge 0,i,j=1,\cdots,n \right \}$.\\

\section{Usual stochastic orders of the finite mixtures }\label{sec-3}
      In this section, we discuss stochastic comparisons of the finite mixture with modified proportional hazard rates model in the sense of the usual stochastic order. Section \ref{sec-3} is divided into two parts, with the first part dealing with stochastic comparison of MPHR mixture models when the mixture proportions $p_{i}$ and tilt parameters $\alpha_{i}$ are different, while the second part deals with MPHR mixture models when the mixture proportions $p_{i}$ and modified proportional hazard rates $\lambda_{i}$ are different. Bring (\ref{1}) into the finite mixture models, the survival function of the finite mixture with MPHR can be expressed as
      $$\bar{F}_{C_{(p,\alpha)}}(x)
      =\sum_{i=1}^{n}p_{i}\frac{\alpha_{i}\bar{F}^{\lambda}(x)}
      {1-\bar{\alpha}_{i}\bar{F}^{\lambda}(x)
      	}
      =\sum_{i=1}^{n}p_{i}\bar{F}_{\alpha_{i}},
      \quad
      0\le \alpha \le 1,$$

      $$\bar{F}_{C_{(p,\lambda)}}(x)
      =\sum_{i=1}^{n}p_{i}\frac{\alpha\bar{F}^{\lambda_{i}}(x)}
      {1-\bar{\alpha}\bar{F}^{\lambda_{i}}(x)}
      =\sum_{i=1}^{n}p_{i}\bar{F}_{\lambda_{i}},
      \quad0\le \alpha \le 1.$$
    
\subsection{Usual stochastic orders of mixture Proportions and Tilt Parameters }\label{sec-3-1} 
    The following Theorems and Corollaries show that two random variables are compared by chain optimization under usual stochastic order when the mixture proportions $p_{i}$ and tilt parameters $\alpha_{i}$ are different.
    
    Theorem \ref{th-1} discusses stochastic comparison of usual stochastic  order in the presence of chain majorization  in two $2\times 2$-order matrices.
\begin{theorem}\label{th-1}  
     Let 
     $\bar{F}_{V_{2}{(p,\alpha)}}(x)
     =\sum_{i=1}^{2}p_{i}
     \frac{\alpha_{i}{\bar{F}^{\lambda }(x) } }
     {1-\bar{\alpha_{i} }\bar{F}^{\lambda}(x)}$ 
     and $\bar{F}_{W_{2}{(q,\beta)}}(x)
     =\sum_{i=1}^{2} q_{i} \frac{\beta_{i}{\bar{F}^{\lambda}(x)}}
     {1-\bar{\beta_{i}}\bar{F}^{\lambda}(x)}$
     be the survival functions of the finite mixtures with modified proportional hazard rates model corresponding to $V_{2}(p,\alpha)$ and $W_{2}(q,\beta)$, respectively.
     
   	\begin{enumerate} [\rm (i)]
    \item  If
    	$\begin{pmatrix}
    		p_{1}& p_{2}\\
    		\alpha_{1}&\alpha_{2}
    	\end{pmatrix}
    	\gg 
    	\begin{pmatrix}
    		q_{1}&q_{2}\\
    		\beta_{1}&\beta_{2}
    	\end{pmatrix}$
    	and $(p,\alpha)\in \ \mathcal {K}_{2}$,
    	then $V_{2}(p,\alpha)\le_{st}W_{2}(q,\beta).$\\
   	 \item If
        $\begin{pmatrix}
        	p_{1}& p_{2}\\
        	\alpha_{1}&\alpha_{2}
        \end{pmatrix}
        \gg 
        \begin{pmatrix}
           q_{1}&q_{2}\\
           \beta_{1}&\beta_{2}
        \end{pmatrix}$,
        $\alpha_{2}p_{1}\bar{F}_{\alpha_{1}}(x)
        \ge
        \alpha_{1}p_{2}\bar{F}_{\alpha_{2}}(x)
        $,
        and $(p,\alpha)\in \ \mathcal{L}_{2}$,
        
        then $V_{2}(p,\alpha)\ge _{st}W_{2}(q,\beta)$.
   \end{enumerate}	
   
\end{theorem}  

\begin{proof} 
	 The survival function of $V_{2}(p,\alpha)$ can be expressed as
	 	$$\bar{F}_{V_{2}{(p,\alpha)}}(x)
	 	=\sum_{i=1}^{2} p_{i} 
	 	\Frac{\alpha _{i} {\bar{F}^{\lambda }(x) } }
	 	{1-\bar{\alpha_{i} }\
	 	\bar{F}^{\lambda}(x)}$$.

	 To establish the desired result, we have to check conditions (i) and (ii) of Lemma \ref{lem-1}. Clearly, $V_{2}(p,\alpha)$ is permutation invariant on $\mathcal{L}_{2}$ and $\mathcal{K}_{2}$, for fixed $x>0$, which confirms Condition (i). Now, for fixed $x>0$ and $i\neq j$, consider the function H
	 \begin{equation}
	 	H(p,\alpha)
	 	=(p_{1}-p_{2})
	 	\left(\frac{\partial \bar{F}_{V_{2}(p,\alpha)}(x)}
	 	{\partial p_{1}}
	 	-\frac{\partial \bar{F}_{V_{2}(p,\alpha)}(x)}
	 	{\partial p_{2}}\right)
	 	+(\alpha_{1}-\alpha_{2})
	 	\left(\frac{\partial\bar{F}_{V_{2}(p,\alpha)}(x)}
	 	{\partial\alpha_{1}}
	 	-\frac{\partial \bar{F}_{V_{2}(p,\alpha)}(x)}
	 	{\partial\alpha_{2}}\right).
	 	  \label{2}
	 \end{equation}
	 
    The partial derivatives of $\bar{F}_{V_{2}(p,\alpha)}(x)$ with respect to $p_{i}$ and $\alpha_{i}$ are\\
	\begin{center} 
	   $\Frac{\partial \bar{F}_{V_{2}(p,\alpha)}(x)}{\partial p_{i}}
	   =\Frac{\alpha_{i}\bar{F}^{\lambda}(x)}
	    {1-\bar{\alpha}_{i}\bar{F}^{\lambda}(x)}$
	   and
	   $\Frac{\partial \bar{F}_{V_{2}(p,\alpha)}(x)}{\partial \alpha_{i}}
	   =\Frac{p_{i}(\bar{F}^{\lambda}(x)-\bar{F}^{2\lambda}(x))}
	    {\left(1-\bar{\alpha}_{i}\bar{F}^{\lambda}(x)\right)^2}$,
	\end{center}
	respectively. Now, upon substituting these expressions in (\ref{2}), we obtain\\\\
	
   ${\Large
   \Frac{\partial \bar{F}_{V_{2}(p,\alpha)}(x)}
   {\partial p_{1}}
   -\Frac{\partial \bar{F}_{V_{2}(p,\alpha)}(x)}
   {\partial p_{2}}
   =
   \Frac{\left(\alpha_{1}-\alpha_{2}\right)
   	\left(\bar{F}^{\lambda}(x)-\bar{F}^{2\lambda}(x)\right)}
   {\left(1-\bar{\alpha}_{1}\bar{F}^{\lambda}(x)\right)
   	\left(1-\bar{\alpha}_{2}\bar{F}^{\lambda}(x)\right)}
   }$,\\  
	and  \\  
	
   ${\Large
   	\Frac{\partial \bar{F}_{V_{2}(p,\alpha)}(x)}
   {\partial \alpha_{1}}
   -\Frac{\partial \bar{F}_{V_{2}(p,\alpha)}(x)}
   {\partial\alpha_{2}}
   =
   \Frac{\left(\bar{F}^{\lambda}(x)-\bar{F}^{2\lambda}(x)\right)
   \left[p_{1}
   \left(1-\bar{\alpha}_{2}\bar{F}^{\lambda}(x)\right)^{2}
	  	-p_{2}
   \left(1-\bar{\alpha}_{1}\bar{F}^{\lambda}(x)\right)^{2}\right]}
   {\left(1-\bar{\alpha}_{1}\bar{F}^{\lambda}(x)\right)^{2}
   \left(1-\bar{\alpha}_{2}\bar{F}^{\lambda}(x)\right)^{2}}
   }$,\\
	then  
	\begin{align*}  
	  H(p,\alpha)
	  &=\frac{(p_{1}-p_{2})(\alpha_{1}-\alpha_{2})
	   \left(\bar{F}^{\lambda}(x)-\bar{F}^{2\lambda}(x)\right)}
	   {\left(1-\bar{\alpha}_{1}\bar{F}^{\lambda}(x)\right)
	   \left(1-\bar{\alpha}_{2}\bar{F}^{\lambda}(x)\right)}\\
	  &\quad+\frac{(\alpha_{1}-\alpha_{2})\left(\bar{F}^{\lambda}(x)-\bar{F}^{2\lambda}(x)\right)
	   \left[p_{1}
	   \left(1-\bar{\alpha}_{2}\bar{F}^{\lambda}(x)\right)^{2}
	  -p_{2}
	   \left(1-\bar{\alpha}_{1}\bar{F}^{\lambda}(x)\right)^{2}\right]}
	   {\left(1-\bar{\alpha}_{1}\bar{F}^{\lambda}(x)\right)^{2}
	   	\left(1-\bar{\alpha}_{2}\bar{F}^{\lambda}(x)\right)^{2}}\\
	  &=\frac{\left(1-\bar{\alpha}_{1}\bar{F}^{\lambda}(x)\right)
	 	\left(1-\bar{\alpha}_{2}\bar{F}^{\lambda}(x)\right)
	 	(p_{1}-p_{2})
	 	(\alpha_{1}-\alpha_{2})
	 	\left(\bar{F}^{\lambda}(x)-\bar{F}^{2\lambda}(x)\right)}
	   {\left(1-\bar{\alpha}_{1}\bar{F}^{\lambda}(x)\right)^{2}
	   	\left(1-\bar{\alpha}_{2}\bar{F}^{\lambda}(x)\right)^{2}}\\
	  &\quad+\frac{(\alpha_{1}-\alpha_{2})
	 	\left(\bar{F}^{\lambda}(x)-\bar{F}^{2\lambda}(x)\right)
	 	\left[p_{1}
	 	\left(1-\bar{\alpha}_{2}\bar{F}^{\lambda}(x)\right)^{2}
	  -p_{2}
	   \left(1-\bar{\alpha}_{1}\bar{F}^{\lambda}(x)\right)^{2}\right]}
	   {\left(1-\bar{\alpha}_{1}\bar{F}^{\lambda}(x)\right)^{2}
	   	\left(1-\bar{\alpha}_{2}\bar{F}^{\lambda}(x)\right)^{2}}\\
	  &\overset{\rm{sgn}}=
	   \left(1-\bar{\alpha}_{1}\bar{F}^{\lambda}(x)\right)
	   \left(1-\bar{\alpha}_{2}\bar{F}^{\lambda}(x)\right)
	   (p_{1}-p_{2})
	   (\alpha_{1}-\alpha_{2})
	   \left(\bar{F}^{\lambda}(x)-\bar{F}^{2\lambda}(x)\right)\\
	  &\quad+(\alpha_{1}-\alpha_{2})
	   \left(\bar{F}^{\lambda}(x)-\bar{F}^{2\lambda}(x)\right)
	   \left[p_{1}
	   \left(1-\bar{\alpha}_{2}\bar{F}^{\lambda}(x)\right)^{2}
	  -p_{2}
	   \left(1-\bar{\alpha}_{1}\bar{F}^{\lambda}(x)\right)^{2}\right]\\
	  &\overset{\rm{sgn}}=
	   \left(\bar{F}^{\lambda}(x)-\bar{F}^{2\lambda}(x)\right)
	   (\alpha_{1}-\alpha_{2})
	   \left(2-\bar{\alpha}_{1}\bar{F}^{\lambda}(x)-\bar{\alpha}_{2}\bar{F}^{\lambda}(x)\right)\\
	  &\quad\times
	   \left[p_{1}
	   \left(1-\bar{\alpha}_{2}\bar{F}^{\lambda}(x)\right)
	  -p_{2}
	   \left(1-\bar{\alpha}_{1}\bar{F}^{\lambda}(x)\right)\right].
	\end{align*}
	
  The assumption that $(p,\alpha)\in\mathcal{K}_{2}$ implies that
    \begin{center}
    	$(p_{1}-p_{2})(\alpha_{1}-\alpha_{2})\le0,$
    \end{center}
     which means that  $p_{1}\ge p_{2}$ and $\alpha_{1}\le \alpha_{2}$ (or $p_{1}\le p_{2}$ and $\alpha_{1}\ge \alpha_{2}$). 
  We present the proof only for tha case when $p_{1}\ge p_{2}$ and $\alpha_{1}\le \alpha_{2}$, since the proof for the other case is quite similar.
  
  Because $\bar{F}(x)$ is a decreasing function with respect to $x$, then
   \begin{center}
   	  $\bar{F}^{\lambda}(x)\ge \bar{F}^{2\lambda}(x).$
   \end{center} 
   
   As $\alpha_{1}\le \alpha_{2}$ and $p_{1}\ge p_{2}$, we readily observe
    \begin{center}
    	 $1-\bar{\alpha}_{2}\bar{F}^{\lambda}(x)\ge 1-\bar{\alpha}_{1}\bar{F}^{\lambda}(x),$
    \end{center}
    and
    \begin{center}
       $p_{1}\left(1-\bar{\alpha}_{2}\bar{F}^{\lambda}(x)\right)
       -p_{2}\left(1-\bar{\alpha}_{1}\bar{F}^{\lambda}(x)\right)\ge0.$
    \end{center}
    
  We obtain 
  \begin{center}
  	$H(p,\alpha)\le 0.$
  \end{center}
  
     Condition (ii) of Lemma \ref{lem-1} is satisfied, implies 
     \begin{center}
     	$V_{2}(p,\alpha)\le_{st} W_{2}(q,\beta)$. 
     \end{center}

  The assumption that $(p,\alpha)\in\mathcal{L}_{2}$ implies that
  \begin{center}
  	$(p_{1}-p_{2})(\alpha_{1}-\alpha_{2})\ge0,$
  \end{center}
   which means that
    $p_{1}\ge p_{2}$ and $\alpha_{1}\ge \alpha_{2}$ (or $p_{1}\le p_{2}$ and $\alpha_{1}\le \alpha_{2}$).
  We present the proof only for tha case when $p_{1}\ge p_{2}$ and $\alpha_{1}\ge \alpha_{2}$. 
  
  Because $\alpha_{1}\ge \alpha_{2}$, then
  \begin{center}
  	$1-\bar{\alpha}_{1}\bar{F}^{\lambda}(x) \ge 
  	1-\bar{\alpha}_{2}\bar{F}^{\lambda}(x).$
  \end{center} 
  
  If
  \begin{center}
  	 $\alpha_{2}p_{1}\bar{F}_{\alpha_{1}}(x)
  	 \ge
  	 \alpha_{1}p_{2}\bar{F}_{\alpha_{2}}(x),$
  \end{center} 
  we have
   \begin{center}
   	$H(p,\alpha)\ge0.$
   \end{center}
   
   Condition (ii) of Lemma \ref{lem-1} is satisfied, implies that 
   \begin{center}
   	$V_{2}(p,\alpha)\ge_{st} W_{2}(q,\beta).$
   \end{center} 
   
   The proof of Theorem \ref{th-1} is completed.
\end{proof}
   
   Example \ref{e-1} provides an illustration of  Theorem \ref{th-1}.
\begin{example}\label{e-1}
	Consider the finite mixture $V_{2}(p,\alpha)$ and $W_{2}(q,\beta)$, where
	$(p,\alpha)\in\mathcal{K}_{2}$,
    $\lambda=0.1$,
	 $(p_{1},p_{2})=(0.6,0.4)$, $(\alpha_{1},\alpha_{2})=(0.3,0.4)$, $(q_{1},q_{2})=(0.48,0.52)$, and $(\beta_{1},\beta_{2})=(0.36,0.34)$.
    We can easily observe that $(q,\beta)=(p,\alpha)T_{0.4}$, which implies $(p,\alpha)\gg(q,\beta)$.
    We take  $\bar{F}(x)=exp\left\{-ax \right \}$ with $a=0.2$. The survival functions of the two finite mixture models are plotted in Figure \ref{1}. It can be seen in Figure \ref{F-1} that $V_{2}(p,\alpha)\le_{st}W_{2}(q,\beta)$.
    Therefore, the validity of Theorem \ref{th-1} is verified.
	
    \begin{figure}[htbp]\label{F-1}
    	\centering
    	\includegraphics[width=0.8\linewidth]{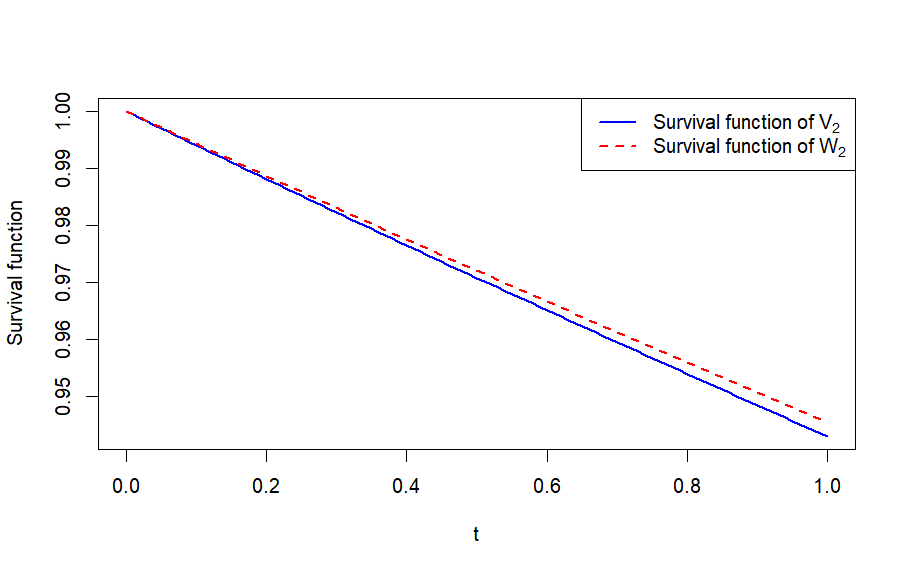}
    	\caption{Plots of $\bar{F}_{V_{2}(p,\alpha)}(x)$ and $\bar{F}_{W_{2}(q,\beta)}(x)$ for $x=t/(1-t)$, $t\in[0,1]$.}
    	\label{fig:1}
    \end{figure}

\end{example}

 Theorem \ref{th-2} discusses stochastic comparison of usual stochastic  order in the presence of chain majorization  in two $n\times n$-order matrices.
\begin{theorem}\label{th-2}  
	Let $\bar{F}_{V_{n}{(p,\alpha)}}(x)=\sum_{i=1}^{n} p_{i} \frac{\alpha _{i} {\bar{F}^{\lambda }(x) } }{1-\bar{\alpha_{i} }\bar{F}^{\lambda}(x)}$
	and $\bar{F}_{W_{n}{(q,\beta)}}(x)=\sum_{i=1}^{n} q_{i} \frac{\beta_{i}{\bar{F}^{\lambda}(x)}}{1-\bar{\beta_{i}}\bar{F}^{\lambda}(x)}$
	be the survival functions of the finite mixtures with modified proportional hazard rates model corresponding to $V_{n}(p,\alpha)$ and $W_{n}(q,\beta)$, respectively.
	
   	\begin{enumerate} [\rm (i)]
	\item  If
	  $\begin{pmatrix}
	  	q_{1}&\cdots&q_{n}\\
		\beta_{1}&\cdots&\beta_{n}
	   \end{pmatrix}
	   =
	   \begin{pmatrix}
	   	p_{1}&\cdots& p_{n}\\
	   	\alpha_{1}&\cdots&\alpha_{n}
	   \end{pmatrix}T$
	and $(p,\alpha)\in \ \mathcal{K}_{n}$,
	then $V_{n}(p,\alpha)\le _{st}W_{n}(q,\beta)$. \\
	\item If
	 $\begin{pmatrix}
		q_{1}&\cdots&q_{n}\\
		\beta_{1}&\cdots&\beta_{n}
	  \end{pmatrix}
	  =
	  \begin{pmatrix}
	  	p_{1}&\cdots& p_{n}\\
	  	\alpha_{1}&\cdots&\alpha_{n}
	  \end{pmatrix}T $,
	 $\alpha_{j}p_{i}\bar{F}_{\alpha_{i}}(x)
	 \ge
	 \alpha_{i}p_{j}\bar{F}_{\alpha_{j}}(x)
	 $,
	and $(p,\alpha)\in \ \mathcal{L}_{n}$,
	then $V_{n}(p,\alpha)\ge_{st}W_{n}(q,\beta)$.
   \end{enumerate}	
\end{theorem}   

\begin{proof}  
	Setting $G_{n}(p,\alpha)=\bar{F}_{V_{n}(p,\alpha)}(x)$ and $H(p,\alpha)=p\alpha\bar{F}^{\lambda}(x)/1-\bar{\alpha}\bar{F}^{\lambda}(x)$, we have $G_{n}(p,\alpha)=\sum_{i=1}^{n}H(p_{i},\alpha_{i})
	=\sum_{i=1}^{n}{p_{i}\alpha_{i}\bar{F}^{\lambda}(x)}/{1-\bar{\alpha}_{i}\bar{F}^{\lambda}(x)}$. According to Theorem \ref{th-1}, $G_{2}(p,\alpha)$ satisfied Lemma \ref{lem-1}. The result required by Theorem \ref{th-2} follows from Lemma \ref{lem-2}. 
\end{proof}

Corollary \ref{c-1} discusses stochastic comparison of usual stochastic  order in the presence of chain majorization  in two $n\times n$-order matrices. In the process of chain optimization, the T-transform matrices has the same structure.
\begin{corollary}\label{c-1}  
	Let $\bar{F}_{V_{n}{(p,\alpha)}}(x)=\sum_{i=1}^{n} p_{i} \frac{\alpha _{i} {\bar{F}^{\lambda }(x) } }{1-\bar{\alpha_{i} }\bar{F}^{\lambda}(x)}$
	and $\bar{F}_{W_{n}{(q,\beta)}}(x)=\sum_{i=1}^{n} q_{i} \frac{\beta_{i}{\bar{F}^{\lambda}(x)}}{1-\bar{\beta_{i}}\bar{F}^{\lambda}(x)}$
	be the survival functions of the finite mixtures with modified proportional hazard rates model corresponding to $V_{n}(p,\alpha)$ and $W_{n}(q,\beta)$, if T-transform matrices $T_{1},...,T_{n}$ have the same structure, respectively.
	
	\begin{enumerate} [\rm (i)]
	\item  If
	 $\begin{pmatrix}
	 	q_{1}&\cdots&q_{n}\\
		\beta_{1}&\cdots&\beta_{n}
	  \end{pmatrix}
	  =
	  \begin{pmatrix}
	  	p_{1}&\cdots& p_{n}\\
	  	\alpha_{1}&\cdots&\alpha_{n}
	  \end{pmatrix}
	  T_{1}\cdots T_{k} $
	and $(p,\alpha)\in \ \mathcal{K}_{n}$,\\
	then $V_{n}(p,\alpha)\le _{st}W_{n}(q,\beta)$.\\ 
	\item If 
	$\begin{pmatrix}
		q_{1}&\cdots&q_{n}\\
		\beta_{1}&\cdots&\beta_{n}
	 \end{pmatrix}
	 =
	 \begin{pmatrix}
	 	p_{1}&\cdots& p_{n}\\
	 	\alpha_{1}&\cdots&\alpha_{n}
	 \end{pmatrix}
	 T_{1}\cdots T_{k}$, 
	 $\alpha_{j}p_{i}\bar{F}_{\alpha_{i}}(x)
	 \ge
	 \alpha_{i}p_{j}\bar{F}_{\alpha_{j}}(x)$,
	and \\
	$(p,\alpha)\in \mathcal{L}_{n}$,
	then $V_{n}(p,\alpha)\ge_{st}W_{n}(q,\beta)$.
	\end{enumerate}	
\end{corollary}

Example \ref{e-2} provides an illustration of result in Corollary \ref{c-1}.

\begin{example}\label{e-2}  
	Consider the finite mixture $V_{n}(p,\alpha)$ and $W_{n}(q,\beta)$, where
	$(p,\alpha)\in\mathcal{K}_{n}$,
	$\lambda=0.2$,
    $(p_{1},p_{2},p_{3})=(0.2,0.3,0.5)$, $(\alpha_{1},\alpha_{2},\alpha_{3})=(0.5,0.3,0.1)$,
	$(q_{1},q_{2},q_{3})=(0.2,0.388,0.412)$,
	and
	$(\beta_{1},\beta_{2},\beta_{3})=(0.5,0.212,0.188)$.
	We take  $\bar{F}(x)=exp\left\{-ax\right\}$ with $a=2$.
	Consider T-transform matrices $T_{1}$ and $T_{2}$ (with the same structure) as follows\\
	
	\begin{center}
		$T_{1}=\begin{pmatrix}
			       1& 0& 0\\
			       0& 0.4& 0.6\\
			       0& 0.6& 0.4
		       \end{pmatrix}$,
	and
	   $T_{2}=\begin{pmatrix}
                  1& 0& 0\\
                  0& 0.2& 0.8\\
                  0& 0.8& 0.2
	   	      \end{pmatrix}$
	\end{center}
    
    \begin{figure}[htbp]\label{F-2}
    	\centering
    	\includegraphics[width=0.8\linewidth]{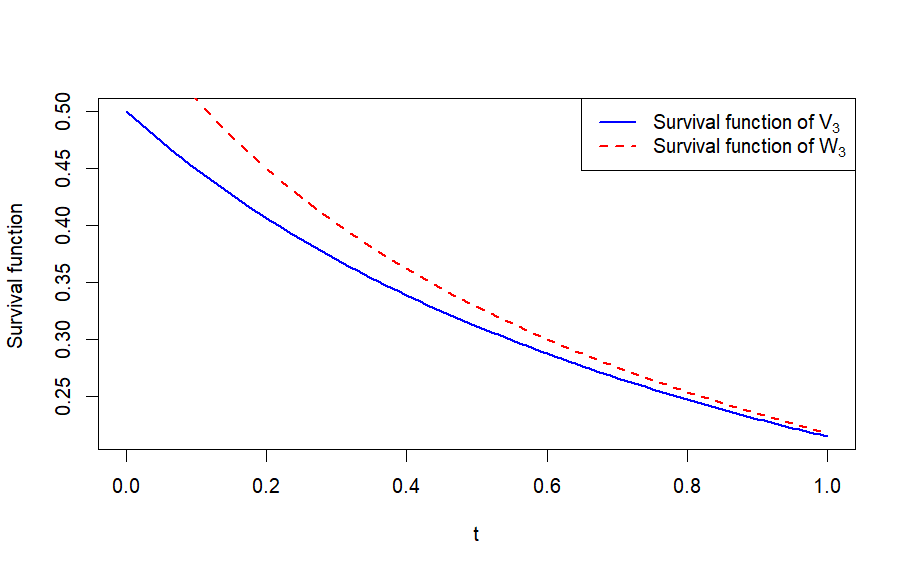}
    	\caption{Plots of $\bar{F}_{V_{3}(p,\alpha)}(x)$ and $\bar{F}_{W_{3}(q,\beta)}(x)$ for $x=t/(1-t)$, $t\in[0,1]$.}
    	\label{fig:2}
    \end{figure}
    
   Then, it is easy to observe that $(q,\beta)=(p,\alpha)T_{1}T_{2}$. 
   From Figure \ref{F-2} it is observed that
   $V_{3}(p,\alpha)\le _{st}W_{3}(q,\beta)$.
   Therefore, the validity of Corollary \ref{c-1} is verified.
\end{example}

Corollary \ref{c-2} discusses stochastic comparison of usual stochastic  order in the presence of chain majorization  in two $n\times n$-order matrices. In the process of chain optimization, the T-transform matrices has the different structure.
\begin{corollary}\label{c-2}  
	Let 
	$\bar{F}_{V_{2}{(p,\alpha)}}(x)=\sum_{i=1}^{n} p_{i} \frac{\alpha _{i} {\bar{F}^{\lambda }(x) } }{1-\bar{\alpha_{i} }\bar{F}^{\lambda}(x)}$
	and 
	$\bar{F}_{W_{2}{(q,\beta)}}(x)=\sum_{i=1}^{n} q_{i} \frac{\beta_{i}{\bar{F}^{\lambda}(x)}}{1-\bar{\beta_{i}}\bar{F}^{\lambda}(x)}$
	be the survival functions of the finite mixtures with modified proportional hazard rates model corresponding to $V_{n}(p,\alpha)$ and $W_{n}(q,\beta)$, if T-transform matrices have the different structure, respectively.
	
	\begin{enumerate} [\rm (i)]
		\item If 
		$\begin{pmatrix}
			q_{1}&\cdots&q_{n}\\
			\beta_{1}&\cdots&\beta_{n}
		 \end{pmatrix}
		 =
		 \begin{pmatrix}
		 	p_{1}&\cdots& p_{n}\\
		 	\alpha_{1}&\cdots&\alpha_{n}
		 \end{pmatrix}
		 T_{1}\cdots T_{k} $,
		 $(p,\alpha)\in \ \mathcal{K}_{n}$ 
		and $(p,\alpha)T_{1}\cdots T_{i} \in\ \mathcal{H}_{n}$, 
		then $V_{n}(p,\alpha)\le _{st}W_{n}(q,\beta)$. \\
		\item If 
		$\begin{pmatrix}
			q_{1}&\cdots&q_{n}\\
			\beta_{1}&\cdots&\beta_{n}
		 \end{pmatrix}
		 =
		 \begin{pmatrix}
		 	p_{1}&\cdots& p_{n}\\
		 	\alpha_{1}&\cdots&\alpha_{n}
		 \end{pmatrix}
		 T_{1}\cdots T_{k} $,
		 $\alpha_{j}p_{i}\bar{F}_{\alpha_{i}}(x)
		 \ge
		 \alpha_{i}p_{j}\bar{F}_{\alpha_{j}}(x)$,
		 $(p,\alpha)\in \ \mathcal{S}_{n}$
		and $(p,\alpha)T_{1}\cdots T_{i}\in\ \mathcal{L}_{n}$,
		then $V_{n}(p,\alpha)\ge_{st}W_{n}(q,\beta)$. \\
	Here $i=1,\cdots,k-1$, where $k\ge2$.
	\end{enumerate}	
    
\end{corollary}

Example \ref{e-3} provides an illustration of result in Corollary \ref{c-2}.
\begin{example} \label{e-3} 
   Consider the finite mixture $V_{n}(p,\alpha)$ and $W_{n}(q,\beta)$, where 
   $(p,\alpha)\in\mathcal{K}_{n}$,
   $\lambda=0.2$,
   $(p_{1},p_{2},p_{3})=(0.1,0.4,0.5)$,
   $(\alpha_{1},\alpha_{2},\alpha_{3})=(0.7,0.5,0.3)$,
   $(q_{1},q_{2},q_{3})=(0.4192,0.248,0.\\3328)$,
   and
   $(\beta_{1},\beta_{2},\beta_{3})=(0.4456,0.568,0.4904)$.
   We take  $\bar{F}(x)=exp\left\{-ax\right\}$ with $a=3$.
   Consider the T-transform matrices $T_{1}$,$T_{2}$, and $T_{3}$ as follows\\
   \begin{center}
   	   $T_{1}=\begin{pmatrix}
   	   	         1& 0&0&\\
   	   	         0& 0.3& 0.7\\
   	   	         0& 0.7& 0.3
   	          \end{pmatrix}$,
   	   $T_{2}=\begin{pmatrix}
   	   	         0.4& 0.6& 0\\
   	   	         0.6& 0.4& 0\\
   	   	         0& 0& 1
   	          \end{pmatrix}$,
   	   and
   	   $T_{3}=\begin{pmatrix}
   	   	         0.1& 0& 0.9\\
   	   	         0& 1& 0\\
   	   	         0.9& 0& 0.1
   	          \end{pmatrix}$   
   \end{center}
   
   \begin{figure}[htbp]\label{F-3}
   	\centering
   	\includegraphics[width=0.8\linewidth]{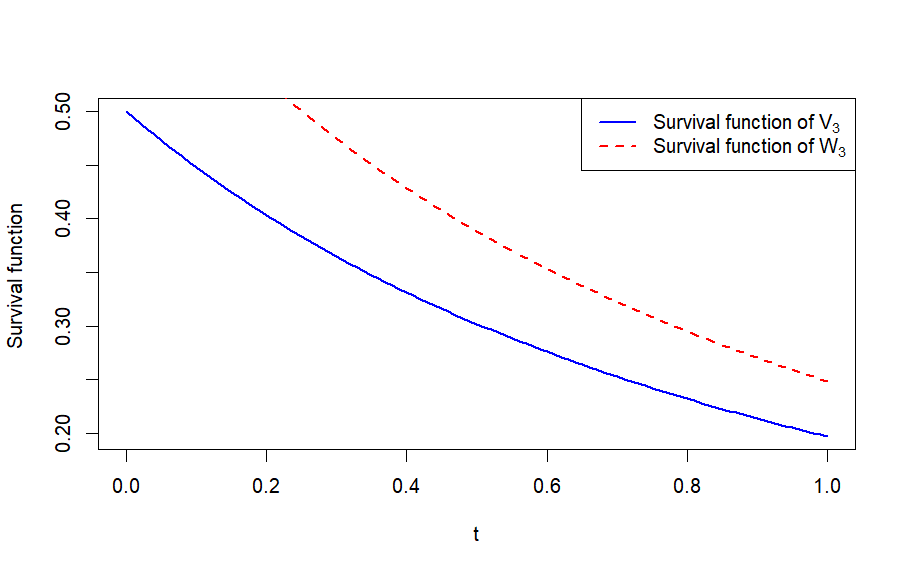}
   	\caption{Plots of $\bar{F}_{V_{3}(p,\alpha)}(x)$ and $\bar{F}_{W_{3}(q,\beta)}(x)$ for $x=t/(1-t)$, $t\in[0,1]$.}
   	\label{fig:3}
   \end{figure}

   It is easy to be observed that
    $(p,\alpha)$, $(p,\alpha)T_{1}$ 
    and  $(p,\alpha)T_{1}T_{2}$ are all in $\mathcal{K}_{n}$, and $(q,\beta)=(p,\alpha)T_{1}T_{2}T_{3}.$
    From Figure \ref{F-3} it is observed that
    $V_{3}(p,\alpha)\le _{st}W_{3}(q,\beta)$.
    Therefore, the validity of Corollary \ref{c-2} is verified.
\end{example}
\subsection{Usual stochastic orders of mixture Proportions and Modified  Proportional Hazard Rate Parameters}\label{sec-3-2} 
    The following Theorems  and Corollaries show that two random variables are compared by chain optimization under usual stochastic order when the mixture proportions $p_{i}$ and modified proportional hazard rate $\lambda_{i}$ are different.
    
    Theorem \ref{th-3} addresses the issue of comparing usual stochastic order when chain majorization is present in $2\times2$-order matrices.
\begin{theorem}\label{th-3}  
	Let 
	$\bar{F}_{Z_{2}(p,\lambda)}(x)=\sum_{i=1}^{2} p_{i} \frac{\alpha {\bar{F}^{\lambda_{i}}(x) } }{1-\bar{\alpha}\bar{F}^{\lambda_{i}}(x)}$
	and $\bar{F}_{Y_{2}(q,\theta)}(x)=\sum_{i=1}^{2} q_{i} \frac{\alpha{\bar{F}^{\theta_{i}}(x)}}{1-\bar{\alpha}\bar{F}^{\theta_{i}}(x)}$
	be the survival functions of the finite mixtures with modified proportional hazard rates model corresponding to $Z_{2}(p,\lambda)$ and $Y_{2}(q,\theta)$, respectively.
	
	\begin{enumerate} [\rm (i)]
		\item If 
		$\begin{pmatrix}
			p_{1}& p_{2}\\
			\lambda_{1}&\lambda_{2}
		 \end{pmatrix}
		 \gg
		 \begin{pmatrix}
		 	q_{1}&q_{2}\\
		 	\theta_{1}&\theta_{2}
		 \end{pmatrix}$
		and $(p,\lambda)\in \ \mathcal{K}_{2}$,
		then $Z_{2}(p,\lambda)\ge _{st}Y_{2}(q,\theta)$. \\
		\item If 
		$\begin{pmatrix}
			p_{1}& p_{2}\\
			\lambda_{1}&\lambda_{2}
		\end{pmatrix}
		\gg
		\begin{pmatrix}
			q_{1}&q_{2}\\
			\theta_{1}&\theta_{2}
		\end{pmatrix}$,
		 $\frac{p_{1}\bar{F}_{\lambda_{1}}(x)}
		 {1-\bar{\alpha}\bar{F}^{\lambda_{1}}(x)}
		 \ge
		 \frac{p_{2}\bar{F}_{\lambda_{2}}(x)}
		 {1-\bar{\alpha}\bar{F}^{\lambda_{2}}(x)}$,
	    and $(p,\lambda)\in\mathcal{L}_{2}$,
		then\\
		$Z_{2}(p,\lambda)\le _{st}Y_{2}(q,\theta)$.
		
	\end{enumerate}	
	
\end{theorem}

\begin{proof}  
	The survival function of $Z_{2}(p,\lambda)$ can be expressed as

	 $$\bar{F}_{Z_{2}(p,\lambda)}(x)=\sum_{i=1}^{2} p_{i} \Frac{\alpha {\bar{F}^{\lambda_{i}}(x) } }{1-\bar{\alpha}\bar{F}^{\lambda_{i}}(x)}$$.

    To establish the desired result, we have to check condition (i) and (ii) of lemma \ref{lem-1}. Clearly, $Z_{2}(p,\lambda)$ is permutation invariant on $\mathcal{L}_{2}$ and $\mathcal{K}_{2}$, for fixed $x>0$, which confirms condition (i). Now, for fixed $x>0$ and $i\neq j$, consider the function H
    
    \begin{equation} 
    	\begin{aligned}
    	H(p,\lambda)&
    	=(p_{1}-p_{2})
    	\left(\frac{\partial\bar{F}_{Z_{2}(p,\lambda)}(x)}
    	{\partial p_{1}}
    	-\frac{\partial\bar{F}_{Z_{2}(p,\lambda)}(x)}
    	{\partial p_{2}}\right)
    	+(\lambda_{1}-\lambda_{2})
    	\left(\frac{\partial\bar{F}_{Z_{2}(p,\lambda)}(x)}
    	{\partial\lambda_{1}}
    	-\frac{\partial\bar{F}_{Z_{2}(p,\lambda)}(x)}
    	{\partial\lambda_{2}}\right).
    	 \end{aligned} 
    	 \label{3}
    \end{equation}
    
    The partial derivatives of $\bar{F}_{Z_{2}(p,\lambda)}(x)$ with respect to $p_{i}$ and $\lambda_{i}$ are\\
    
    \begin{center}   
    	$\Frac{\partial\bar{F}_{Z_{2}(p,\lambda)}(x)}
    	{\partial p_{i}}
    	=\Frac{\alpha\bar{F}^{\lambda_{i}}(x)}
    	{1-\bar{\alpha}\bar{F}^{\lambda_{i}}(x)}$
    and
        $\Frac{\partial\bar{F}_{Z_{2}(p,\lambda)}(x)}
        {\partial\lambda_{i}}
        =\Frac{\alpha p_{i}\bar{F}^{\lambda_{i}}(x)\ln\bar{F}(x)}
        {\left(1-\bar{\alpha}\bar{F}^{\lambda_{i}}(x)\right)^{2}}$,
    \end{center}
    respectively. Now, upon substituting these expressions in (\ref{3}), we obtain\\
    
    $\Frac{\partial\bar{F}_{Z_{2}(p,\lambda)}(x)}{\partial p_{1}}
    -\Frac{\partial\bar{F}_{Z_{2}(p,\lambda)}(x)}{\partial p_{2}}
    =
    \Frac{\alpha(\bar{F}^{\lambda_{1}}(x)
    	-\bar{F}^{\lambda_{2}}(x))}
    {\left(1-\bar\alpha\bar{F}^{\lambda_{1}}(x)\right)
    	\left(1-\bar\alpha\bar{F}^{\lambda_{2}}(x)\right)}$,\\
    and\\	
   
   $\Frac{\partial\bar{F}_{Z_{2}(p,\lambda)}(x)}{\partial\lambda_{1}}
   -\Frac{\partial\bar{F}_{Z_{2}(p,\lambda)}(x)}{\partial\lambda_{2}}
   =
   \Frac{\alpha \ln\bar{F}(x)\left[	p_{1}\bar{F}^{\lambda_{1}}(x)
   	\left(1-\bar{\alpha}\bar{F}^{\lambda_{2}}(x)\right)^{2}
   	-
   	p_{2}\bar{F}^{\lambda_{2}}(x)
   	\left(1-\bar{\alpha}\bar{F}^{\lambda_{1}}(x)\right)^{2}\right]
   }
   {\left(1-\bar{\alpha}\bar{F}^{\lambda_{1}}(x)\right)^{2}
   	\left(1-\bar{\alpha}\bar{F}^{\lambda_{2}}(x)\right)^{2}}
   $,\\
   	then
    \begin{equation}
      \begin{aligned}   
    	H(p,\lambda)
    	&=(p_{1}-p_{2})
    	\left(\frac{\partial\bar{F}_{Z_{2}(p,\lambda)}(x)}
    	{\partial p_{1}}
    	-\frac{\partial\bar{F}_{Z_{2}(p,\lambda)}(x)}
    	{\partial p_{2}}\right)\\
    	&\quad+(\lambda_{1}-\lambda_{2})
    	\left(\frac{\partial\bar{F}_{Z_{2}(p,\lambda)}(x)}
    	{\partial\lambda_{1}}
    	-\frac{\partial\bar{F}_{Z_{2}(p,\lambda)}(x)}
    	{\partial\lambda_{2}}\right) \\
    	&=\frac{\alpha(p_{1}-p_{2})
    		\left(\bar{F}^{\lambda_{1}}(x)-\bar{F}^{\lambda_{2}}(x)\right)}
    	{\left(1-\bar\alpha\bar{F}^{\lambda_{1}}(x)\right)
    		\left(1-\bar\alpha\bar{F}^{\lambda_{2}}(x)\right)}\\
    	&\quad+\Frac{\alpha\ln\bar{F}(x)
    		(\lambda_{1}-\lambda_{2})
    		\left[	p_{1}\bar{F}^{\lambda_{1}}(x)
    		\left(1-\bar{\alpha}\bar{F}^{\lambda_{2}}(x)\right)^{2}
    		-
    		p_{2}\bar{F}^{\lambda_{2}}(x)
    		\left(1-\bar{\alpha}\bar{F}^{\lambda_{1}}(x)\right)^{2}\right]
    	}
    	{\left(1-\bar{\alpha}\bar{F}^{\lambda_{1}}(x)\right)^{2}
    		\left(1-\bar{\alpha}\bar{F}^{\lambda_{2}}(x)\right)^{2}}\\
    	&=\frac{(p_{1}-p_{2})
    		\alpha
    		\left(\bar{F}^{\lambda_{1}}(x)-\bar{F}^{\lambda_{2}}(x)\right)
    		\left(1-\bar\alpha\bar{F}^{\lambda_{1}}(x)\right)
    		\left(1-\bar\alpha\bar{F}^{\lambda_{2}}(x)\right)}
    	{\left(1-\bar{\alpha}\bar{F}^{\lambda_{1}}(x)\right)^{2}
    		\left(1-\bar{\alpha}\bar{F}^{\lambda_{2}}(x)\right)^{2}}\\
    	&\quad+\frac{\alpha\ln\bar{F}(x) p_{1}\bar{F}^{\lambda_{1}}(x)
    		(\lambda_{1}-\lambda_{2})
    		\left(1-\bar{\alpha}\bar{F}^{\lambda_{2}}(x)\right)^{2}}
    	{\left(1-\bar{\alpha}\bar{F}^{\lambda_{1}}(x)\right)^{2}
    		\left(1-\bar{\alpha}\bar{F}^{\lambda_{2}}(x)\right)^{2}}\\
    	&\quad-\frac{\alpha\ln\bar{F}(x) p_{2}\bar{F}^{\lambda_{2}}(x)(\lambda_{1}-\lambda_{2})
    		\left(1-\bar{\alpha}\bar{F}^{\lambda_{1}}(x)\right)^{2}}
    	{\left(1-\bar{\alpha}\bar{F}^{\lambda_{1}}(x)\right)^{2}
    		\left(1-\bar{\alpha}\bar{F}^{\lambda_{2}}(x)\right)^{2}}\\
    	&\overset{\rm{sgn}}=
    	(p_{1}-p_{2})
    	\alpha
    	\left(\bar{F}^{\lambda_{1}}(x)-\bar{F}^{\lambda_{2}}(x)\right)
    	\left(1-\bar\alpha\bar{F}^{\lambda_{1}}(x)\right)
    	\left(1-\bar\alpha\bar{F}^{\lambda_{2}}(x)\right)\\
    	&\quad+\alpha\ln\bar{F}(x) p_{1}\bar{F}^{\lambda_{1}}(x)(\lambda_{1}-\lambda_{2})
    	\left(1-\bar{\alpha}\bar{F}^{\lambda_{2}}(x)\right)^{2}\\
    	&\quad-\alpha\ln\bar{F}(x) p_{2}\bar{F}^{\lambda_{2}}(x)(\lambda_{1}-\lambda_{2})
    	\left(1-\bar{\alpha}\bar{F}^{\lambda_{1}}(x)\right)^{2}\\
    	&\overset{\rm{sgn}}=
    	(p_{1}-p_{2})
    	\left(\bar{F}^{\lambda_{1}}(x)-\bar{F}^{\lambda_{2}}(x)\right)
    	\left(1-\bar\alpha\bar{F}^{\lambda_{1}}(x)\right)
    	\left(1-\bar\alpha\bar{F}^{\lambda_{2}}(x)\right)\\
    	&\quad+\ln\bar{F}(x)(\lambda_{1}-\lambda_{2})
    	\left[p_{1}\bar{F}^{\lambda_{1}}(x)
    	\left(1-\bar{\alpha}\bar{F}^{\lambda_{2}}(x)\right)^{2}
    	-p_{2}\bar{F}^{\lambda_{2}}(x)
    	\left(1-\bar{\alpha}\bar{F}^{\lambda_{1}}(x)\right)^{2}\right].
    	\qquad  \qquad  \qquad  \qquad  \qquad  (4)
       \end{aligned}
     \label{4}
    \end{equation}

	The assumption that $(p,\lambda)\in\mathcal{K}_{2}$ implies that
    \begin{center}
		$(p_{1}-p_{2})(\lambda_{1}-\lambda_{2})\le0,$
	\end{center}
	which means that  $p_{1}\ge p_{2}$ and $\lambda_{1}\le\lambda_{2}$ (or $p_{1}\le p_{2}$ and $\lambda_{1}\ge \lambda_{2}$). 
	We present the proof only for the case when $p_{1}\ge p_{2}$ and $\lambda_{1} \le\lambda_{2}$, since the proof for the other case is quite similar. 
	
	Because $\bar{F}(x)$ is a decrease function with respect to $x$, then
	 \begin{center}
	 	$\bar{F}^{\lambda_{1}}(x)\ge\bar{F}^{\lambda_{2}}(x)$,
	 \end{center} 
	we readily observe that the first term on the right-hand of (\ref{3}) is positive.
	
    Moreover, according to the assumption that $\bar{F}(x)$ is within the scope of $[0,1]$ implies \begin{center}
    	$\ln\bar{F}(x)\le0$, 
    \end{center}
	the second term on the right -hand side of (\ref{4}) is also positive. Upon combining these observations, we have \begin{center}
		$Z_{2}(p,\lambda)\ge_{st} Y_{2}(q,\theta)$.
	\end{center}
	
	The assumption that $(p,\lambda)\in\mathcal{L}_{2}$ implies that \begin{center}
		$(p_{1}-p_{2})(\lambda_{1}-\lambda_{2})\ge0$,
	\end{center}
	 which means that $p_{1}\ge p_{2}$ and $\lambda_{1}\ge\lambda_{2}$ or ($p_{1}\le p_{2}$ and $\lambda_{1}\le \lambda_{2}$). 
	 We present the proof only for the case when $p_{1}\ge p_{2}$ and $\lambda_{1} \ge\lambda_{2}$, 
	 since the proof for the other case is quite similar. In this case, we have 
	 \begin{center}
	 	$\bar{F}^{\lambda_{1}}(x)\le \bar{F}^{\lambda_{2}}(x)$, 
	 \end{center}
	 the first term on the right-hand side of (\ref{4}) is non-positive. If $\frac{p_{1}\bar{F}_{\lambda_{1}}(x)}
	 {1-\bar{\alpha}\bar{F}^{\lambda_{1}}(x)}
	 \ge\frac{p_{2}\bar{F}_{\lambda_{2}}(x)}
	 {1-\bar{\alpha}\bar{F}^{\lambda_{2}}(x)}$, we have
	  \begin{center}
	 	$H(p,\lambda)\le0$.
	 \end{center} 
	 
	Condition (ii) of Lemma \ref{lem-1} is satisfied, 
	 implies that $Z_{2}(p,\lambda)\le_{st} Y_{2}(q,\theta)$. The proof of Theorem \ref{th-3} is completed.
\end{proof}
    Example \ref{e-4} provides an illustration of  Theorem \ref{th-3}.
\begin{example}\label{e-4} 
	Consider the finite mixture $Z_{2}(p,\lambda)$ and $Y_{2}(q,\theta)$, where $(p,\lambda)\in\mathcal{K}_{2}$,
    $\alpha=0.2$,
    $(p_{1},p_{2})=(0.2,0.8)$, 
	$(\lambda_{1},\lambda_{2})=(0.5,0.25)$,
	$(q_{1},q_{2})=(0.62,0.38)$
	and $(\theta_{1},\theta_{2})=(0.325,0.425)$.
	It can be observed that $(q,\theta)=(p,\lambda)T_{0.3}$,
	implies $(p,\lambda)\gg(q,\theta)$.
	We take $\bar{F}(x)=exp\left\{-ax\right\}$
	with $a=2$.
	The survival functions of the two finite mixture models are plotted in Figure \ref{F-4}. 
	It can be seen in Figure \ref{F-4} that $Z_{2}(p,\lambda)\ge _{st}Y_{2}(q,\theta)$.
	Therefore, the validity of Theorem \ref{th-3} has been confirmed.
	
	\begin{figure}[htbp]\label{F-4}
		\centering
		\includegraphics[width=0.8\linewidth]{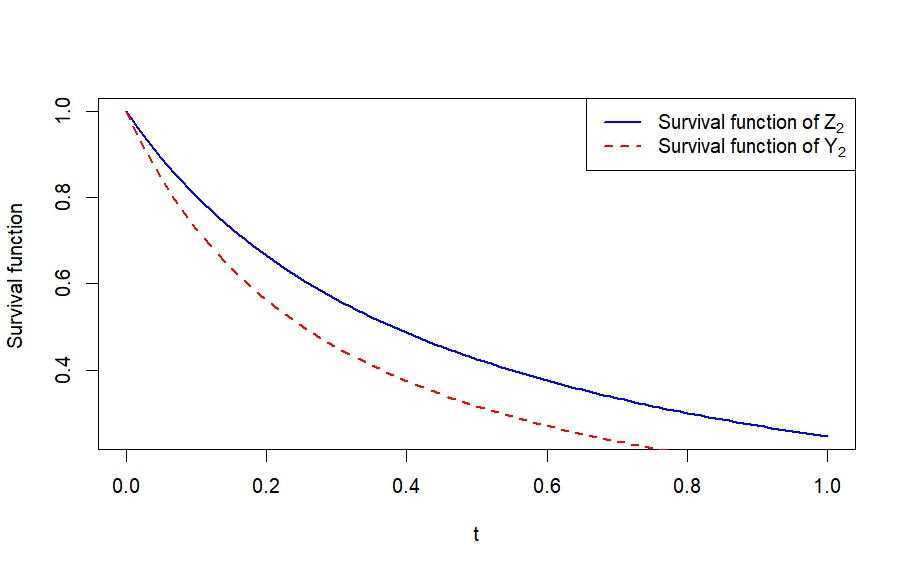}
		\caption{Plots of $\bar{F}_{Z_{2}(p,\lambda)}(x)$ and $\bar{F}_{Y_{2}(q,\theta)}(x)$ for $x=t/(1-t)$, $t\in[0,1]$.}
		\label{fig:4}
	\end{figure}
	
\end{example}

Theorem \ref{th-4} addresses the issue of comparing usual stochastic order when chain majorization is present in $n\times n$-order matrices.
\begin{theorem}\label{th-4} 
	Let
	$\bar{F}_{Z_{n}{(p,\lambda)}}(x)=\sum_{i=1}^{n} p_{i} \frac{\alpha {\bar{F}^{\lambda_{i}}(x) } }{1-\bar{\alpha}\bar{F}^{\lambda_{i}}(x)}$
	and $\bar{F}_{Y_{n}{(q,\theta)}}(x)=\sum_{i=1}^{n} q_{i} \frac{\alpha{\bar{F}^{\theta_{i}}(x)}}{1-\bar{\alpha}\bar{F}^{\theta_{i}}(x)}$
	be the survival functions of the finite mixtures with modified proportional hazard rates model corresponding to $Z_{n}(p,\lambda)$ and $Y_{n}(q,\theta)$, respectively.
	\begin{enumerate} [\rm (i)]
		\item If 
		$\begin{pmatrix}
			q_{1}&\cdots&q_{n}\\
			\theta_{1}&\cdots&\theta_{n}
		 \end{pmatrix}
		=
		\begin{pmatrix}
			p_{1}&\cdots& p_{n}\\
			\lambda_{1}&\cdots&\lambda_{n}
		\end{pmatrix}T$
		and $(p,\lambda)\in \ \mathcal{K}_{n}$,
		then $Z_{n}(p,\lambda)\ge _{st}Y_{n}(q,\theta)$. \\
		\item If
		 $\begin{pmatrix}
			q_{1}&\cdots&q_{n}\\
			\theta_{1}&\cdots&\theta_{n}
		  \end{pmatrix}
		 =
		  \begin{pmatrix}
		  	p_{1}&\cdots&p_{n}\\
			\lambda_{1}&\cdots&\lambda_{n}
		  \end{pmatrix}T$,
		$\frac{p_{i}\bar{F}_{\lambda_{i}}(x)}
		{1-\bar{\alpha}\bar{F}^{\lambda_{i}}(x)}
		\ge
		\frac{p_{j}\bar{F}_{\lambda_{j}}(x)}
		{1-\bar{\alpha}\bar{F}^{\lambda_{j}}(x)}$,
		and $(p,\lambda)\in\mathcal{L}_{n}$,\\
		then $Z_{n}(p,\lambda)\le _{st}Y_{n}(q,\theta)$.	
	\end{enumerate}	
\end{theorem}
\begin{proof} 
	Setting $G_{n}(p,\lambda)=\bar{F}_{Z_{n}(p,\lambda)}(x)$ and $H(p,\lambda)=p\alpha\bar{F}^{\lambda}(x)/1-\bar{\alpha}\bar{F}^{\lambda}(x)$, we have $G_{n}(p,\lambda)=\sum_{i=1}^{n}H(p_{i},\lambda_{i})
	=\sum_{i=1}^{n}{p_{i}\alpha\bar{F}^{\lambda_{i}}(x)}/{1-\bar{\alpha}\bar{F}^{\lambda_{i}}(x)}$. According to Theorem \ref{th-1}, $G_{2}(p,\lambda)$ satisfied Lemma \ref{lem-1}. The result required by Theorem \ref{th-4} follows from Lemma \ref{lem-2}.  
\end{proof}

Corollary \ref{c-3} addresses the issue of comparing usual stochastic order when chain majorization is present in $n\times n$-order matrices. In the context of chain majorization, the T-transform matrices exhibits the same structure.
\begin{corollary}\label{c-3}   
	Let
	 $\bar{F}_{Z_{n}{(p,\lambda)}}(x)=\sum_{i=1}^{n} p_{i} \frac{\alpha {\bar{F}^{\lambda_{i}}(x) } }{1-\bar{\alpha}\bar{F}^{\lambda_{i}}(x)}$
	and $\bar{F}_{Y_{n}{(q,\theta)}}(x)=\sum_{i=1}^{n} q_{i} \frac{\alpha{\bar{F}^{\theta_{i}}(x)}}{1-\bar{\alpha}\bar{F}^{\theta_{i}}(x)}$
	be the survival functions of the finite mixtures with modified proportional hazard rates model corresponding to $Z_{n}(p,\lambda)$ and $Y_{n}(q,\theta)$, if T-transform matrices $T_{1},...,T_{n}$ have the same structure, respectively.
		
	\begin{enumerate} [\rm (i)]
		\item If 
		$\begin{pmatrix}
			q_{1}&\cdots&q_{n}\\
			\theta_{1}&\cdots&\theta_{n}
		 \end{pmatrix}
		=
		\begin{pmatrix}
			p_{1}&\cdots& p_{n}\\
			\lambda_{1}&\cdots&\lambda_{n}
		\end{pmatrix}
		T_{1}\cdots T_{k}$
		and $(p,\lambda)\in \ \mathcal{K}_{n}$,
		then\\ 
		$Z_{n}(p,\lambda)\ge _{st}Y_{n}(q,\theta)$.\\
		\item If 
		$\begin{pmatrix}
			q_{1}&\cdots&q_{n}\\
			\theta_{1}&\cdots&\theta_{n}
		 \end{pmatrix}
		=
		\begin{pmatrix}
			p_{1}&\cdots&p_{n}\\
			\lambda_{1}&\cdots&\lambda_{n}
		\end{pmatrix}
		T_{1}\cdots T_{k}$,
		$\frac{p_{i}\bar{F}_{\lambda_{i}}(x)}
		{1-\bar{\alpha}\bar{F}^{\lambda_{i}}(x)}
		\ge
		\frac{p_{j}\bar{F}_{\lambda_{j}}(x)}
		{1-\bar{\alpha}\bar{F}^{\lambda_{j}}(x)}$,
		and
		$(p,\lambda)\in\mathcal{L}_{n}$,
		then $Z_{n}(p,\lambda)\le _{st}Y_{n}(q,\theta)$.	
	\end{enumerate}	
\end{corollary}	

Example \ref{e-5} provides an illustration of Corollary \ref{c-3}.

\begin{example}\label{e-5} 
	Consider the mixture $Z_{n}(p,\lambda)$ and $Y_{n}(q,\theta)$, where $(p,\lambda)\in\mathcal{K}_{n}$,
	$\alpha=0.2$,
	Suppose $(p_{1},p_{2},p_{3})=(0.5,0.4,0.1)$, $(\lambda_{1},\lambda_{2},\lambda_{3})=(3,4,5)$,
	$(q_{1},q_{2},q_{3})=(0.5,0.268,0.232)$,
	and
	$(\beta_{1},\beta_{2},\beta_{3})=(3,4.44,4.56)$.
	We take $\bar{F}(x)=exp\left\{-ax\right\}$
	with $a=0.2$.
	Consider T-transform matrices $T_{1}$ and $T_{2}$ (with the same structure) as follows\\
	
	\begin{center}
		$T_{1}=\begin{pmatrix}
			1& 0& 0\\
			0& 0.4& 0.6\\
			0& 0.6& 0.4
		\end{pmatrix}$,
		and 
		$T_{2}=\begin{pmatrix}
			1& 0& 0\\
			0& 0.2& 0.8\\
			0& 0.8& 0.2
		\end{pmatrix}$
	\end{center}
	\begin{figure}[htbp]\label{F-5}
		\centering
		\includegraphics[width=0.8\linewidth]{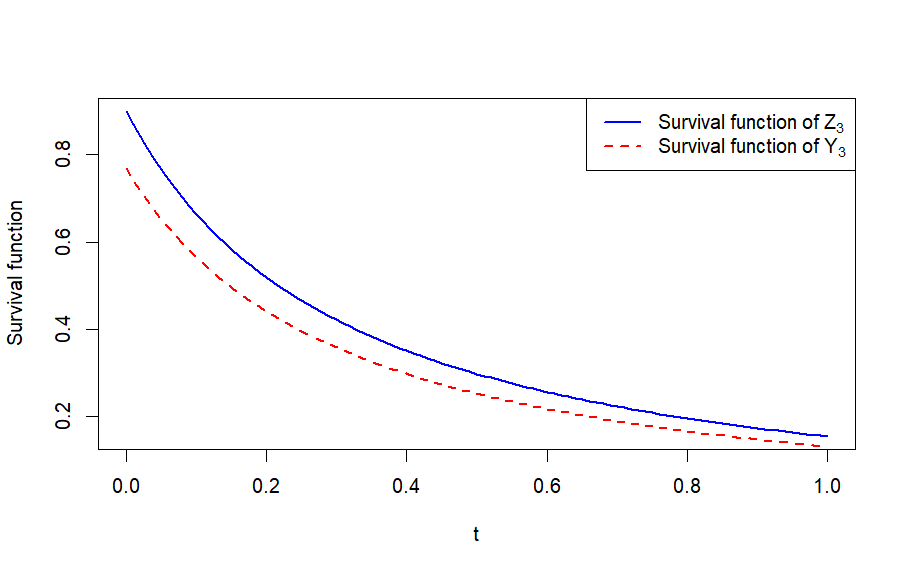}
		\caption{Plots of $\bar{F}_{Z_{3}(p,\lambda)}(x)$ and $\bar{F}_{Y_{3}(q,\theta)}(x)$ for $x=t/(1-t)$, $t\in[0,1]$.}
		\label{fig:5}
	\end{figure}
	
	Then, it is easy to be observed that $(q,\beta)=(p,\alpha)T_{1}T_{2}$, From Figure \ref{F-5} it is observed that
    $Z_{3}(p,\lambda)\ge _{st}Y_{3}(q,\theta)$. 
    Therefore, the validity of Corollary \ref{c-3} has been confirmed.
\end{example}

Corollary \ref{c-4} addresses the issue of comparing usual stochastic order when chain majorization is present in $n\times n$-order matrices. In the context of chain majorization, the T-transform matrices exhibits the different structure.
\begin{corollary}\label{c-4}   
	Let
	 $\bar{F}_{Z_{n}{(p,\lambda)}}(x)=\sum_{i=1}^{n} p_{i} \frac{\alpha {\bar{F}^{\lambda_{i}}(x) } }{1-\bar{\alpha}\bar{F}^{\lambda_{i}}(x)}$
	and $\bar{F}_{Y_{n}{(q,\theta)}}(x)=\sum_{i=1}^{n} q_{i} \frac{\alpha{\bar{F}^{\theta_{i}}(x)}}{1-\bar{\alpha}\bar{F}^{\theta_{i}}(x)}$
	be the survival functions of the finite mixtures with modified proportional hazard rates model corresponding to $Z_{n}(p,\lambda)$ and $Y_{n}(q,\theta)$, if T-transform matrices have the different structure, respectively.
	
	\begin{enumerate} [\rm (i)]
		\item If 
		$\begin{pmatrix}
			q_{1}&\cdots&q_{n}\\
			\theta_{1}&\cdots&\theta_{n}
		 \end{pmatrix}
		=
		\begin{pmatrix}
			p_{1}&\cdots& p_{n}\\
			\lambda_{1}&\cdots&\lambda_{n}
		\end{pmatrix}
		T_{1}\cdots T_{k} $,
		$(p,\lambda)\in \ \mathcal{K}_{n}$,
		and $(p,\lambda)T_{1}\cdots T_{i}\in  \mathcal{S}_{n}$, 
		then $Z_{n}(p,\lambda)\ge _{st}Y_{n}(q,\theta)$.
\\
		\item If 
		$\begin{pmatrix}
			q_{1}&\cdots&q_{n}\\
			\theta_{1}&\cdots&\theta_{n}
		 \end{pmatrix}
		=
		\begin{pmatrix}
			p_{1}&\cdots&p_{n}\\
			\lambda_{1}&\cdots&\lambda_{n}
		\end{pmatrix}
		T_{1}\cdots T_{k}$,
		$\frac{p_{i}\bar{F}_{\lambda_{i}}(x)}
		{1-\bar{\alpha}\bar{F}^{\lambda_{i}}(x)}
		\ge
		\frac{p_{j}\bar{F}_{\lambda_{j}}(x)}
		{1-\bar{\alpha}\bar{F}^{\lambda_{j}}(x)}$,
		$(p,\lambda)\in\mathcal{L}_{n}$,
		and
		$(p,\lambda)T_{1}\cdots T_{i}\in\mathcal{H}_{n}$,
		then $Z_{n}(p,\lambda)\le _{st}Y_{n}(q,\theta)$. \\
	Here $i=1,\cdots,k-1$, where $k\ge2$.
	\end{enumerate}	
	
\end{corollary}
\begin{remark} 
	In this section, we established the new finite mixture, what baseline distribution is modified proportional hazard rate model. 
	The new finite mixture model uses the same research methodology as \cite{barmalzan2022orderings}, except that the baseline distribution in the new mixture model is one model (i.e., MPHR) and the baseline distribution in \cite{barmalzan2022orderings} is a distribution.
\end{remark}

\section{Hazard rate orders of the finite mixtures}\label{sec-4} 
    In this section, we discuss stochastic comparisons of the finite mixture with modified proportional hazard rates model in the sense of the hazard rate order. 
    The following Theorems and Corollaries show that two random variables are compared by chain optimization under hazard rate order when the mixture proportions $p_{i}$ and tilt parameters $\alpha_{i}$ are different.
    The survival function of the finite mixture with MPHR can be expressed as
    $$\bar{F}_{C_{(p,\alpha)}}(x)
    =\sum_{i=1}^{n}p_{i}\frac{\alpha_{i}\bar{F}^{\lambda}(x)}
    {1-\bar{\alpha}_{i}\bar{F}^{\lambda}(x)},
    \quad
    0\le \alpha \le 1.$$
   Theorem \ref{th-5}  shows how to compare the hazard rate orders of finite mixture models for two $2\times 2$-order matrices under chain majorization.
\begin{theorem} \label{th-5} 
	 Let 
	$\bar{F}_{V_{2}{(p,\alpha)}}(x)=\sum_{i=1}^{2}p_{i}\frac{\alpha_{i}{\bar{F}^{\lambda }(x)}}{1-\bar{\alpha_{i} }\bar{F}^{\lambda}(x)}$ 
	and $\bar{F}_{W_{2}{(q,\beta)}}(x)=\sum_{i=1}^{2} q_{i} \frac{\beta_{i}{\bar{F}^{\lambda}(x)}}{1-\bar{\beta_{i}}\bar{F}^{\lambda}(x)}$
	be the survival functions of the finite mixtures with modified proportional hazard rates model corresponding to $V_{2}(p,\alpha)$ and $W_{2}(q,\beta)$. Further, suppose $p_{1}\alpha_{1}=p_{2}\alpha_{2}$ and $r(x)$ is positive, then for $(p,\alpha)\in\mathcal{K}_{2}$, respectively.
	\begin{center}
		$\begin{pmatrix}
			p_{1}& p_{2}\\
			\alpha_{1}&\alpha_{2}
		 \end{pmatrix}
		 \gg 
		 \begin{pmatrix}
		    q_{1}&q_{2}\\
		    \beta_{1}&\beta_{2}
		 \end{pmatrix}
		\Longrightarrow
		  V_{2}(p,\alpha)\ge_{hr}W_{2}(q,\beta).$
	\end{center}
\end{theorem}
\begin{proof}  
  The probability density function of $V_{2}(p,\alpha)$ can be given by
 	   $$f_{V_{2}(p,\alpha)}(x)
	   =\sum_{i=1}^{2}p_{i}
	   \frac{\lambda\alpha_{i}\bar{F}^{\lambda-1}(x)f(x)}
	   {\left(1-\bar{\alpha}_{i}\bar{F}^{\lambda}(x)\right)^{2}}$$,

  Let $M_{1}=1-\bar{\alpha}_{1}\bar{F}^{\lambda}(x)$ and $M_{2}=1-\bar{\alpha}_{2}\bar{F}^{\lambda}(x)$, the hazard rate function of $V_{2}(p,\alpha)$ can be expressed as
	\begin{align*}  
      r_{V_{2}(p,\alpha)}(x)
      &=\Frac{f_{V_{2}(p,\alpha)}(x)}
       {\bar{F}_{V_{2}(p,\alpha)}(x)}
      =\Frac{\sum_{i=1}^{2}p_{i}
       \Frac{\lambda\alpha_{i}\bar{F}^{\lambda-1}(x)f(x)}
       {\left(1-\bar{\alpha}_{i}\bar{F}^{\lambda}(x)\right)^{2}}}
       {{\sum_{i=1}^{2}}p_{i}
       \Frac{\alpha_{i}{\bar{F}^{\lambda }(x)}}
       {1-\bar{\alpha_{i} }\bar{F}^{\lambda}(x)}}\\
      &=\frac{\lambda r(x)p_{1}\alpha_{1}
       \left(1-\bar{\alpha}_{2}\bar{F}^{\lambda}(x)\right)^{2}
      +p_{2}\alpha_{2}
       \left(1-\bar{\alpha}_{1}\bar{F}^{\lambda}(x)\right)^{2}}
       {p_{1}\alpha_{1}
       \left(1-\bar{\alpha}_{1}\bar{F}^{\lambda}(x)\right)
       \left(1-\bar{\alpha}_{2}\bar{F}^{\lambda}(x)\right)^{2}
      +p_{2}\alpha_{2}
       \left(1-\bar{\alpha}_{1}\bar{F}^{\lambda}(x)\right)^{2}
       \left(1-\bar{\alpha}_{2}\bar{F}^{\lambda}(x)\right)}.
	\end{align*}
  
  Then
    \begin{center} 
       $r_{V_{2}(p,\alpha)}(x)=
       \Frac{\lambda r(x)p_{1}\alpha_{1}M^{2}_{2}
  	   +p_{2}\alpha_{2}M^{2}_{1}}
  	   {p_{1}\alpha_{1}M_{1}M^{2}_{2}
  	   +p_{2}\alpha_{2}M^{2}_{1}M_{2}}$.
    \end{center}
    
  To establish the desired result, we have to check Conditions (i) and (ii) of Lemma \ref{lem-1}. Clearly, $r_{V_{2}(p,\alpha)}(x)$ is permutation invariant on $\mathcal{K}_{2}$, for fixed $x>0$, which confirms Condition (i). Next, for fixed $x>0$ and $i\neq j$, consider the function
  
 \begin{equation}    
       H(p,\alpha)
       =(p_{1}-p_{2})
       \left(\frac{\partial r_{V_{2}(p,\alpha)}(x)}
       {\partial p_{1}}
       -\frac{\partial r_{V_{2}(p,\alpha)}(x)}
       {\partial p_{2}}\right)
       +(\alpha_{1}-\alpha_{2})
       \left(\frac{\partial r_{V_{2}(p,\alpha)}(x)}
       {\partial \alpha_{1}}
       -\frac{\partial r_{V_{2}(p,\alpha)}(x)}
       {\partial \alpha_{2}}\right) .
  \label{5}
  \end{equation}
  
  The partial derivatives of $r_{V_{2}(p,\alpha)}(x)$ with respect to $p_{1}$, $p_{2}$, $\alpha_{1}$ and $\alpha_{2}$ are\\  
  
  $\Frac{\partial r_{V_{2}(p,\alpha)}(x)} {\partial p_{1}}
  =\Frac{\lambda r(x)p_{1}\alpha_{1}\alpha_{2}M_{1}M_{2}\left(M^{3}_{2}-M^{3}_{1}\right)}
  {\left(p_{1}\alpha_{1}M^{2}_{1}M_{2}
  +p_{2}\alpha_{2}M_{1}M^{2}_{2}\right)^{2}}$, \\
   and \\
   
  $\Frac{\partial r_{V_{2}(p,\alpha)}(x)}  {\partial p_{2}}
  =\Frac{\lambda r(x)p_{2}\alpha_{1}\alpha_{2}M_{1}M_{2}
  \left(M^{3}_{1}-M^{3}_{2}\right)}
  {\left(p_{1}\alpha_{1}M^{2}_{1}M_{2}
  +p_{2}\alpha_{2}M_{1}M^{2}_{2}\right)^{2}}$, \\
 or\\
 
  $\Frac{\partial r_{V_{2}(p,\alpha)}(x)}   {\partial \alpha_{1}}
  =\Frac{\lambda r(x)p_{1}p_{2}\alpha_{2}M_{1}M_{2}
   \left(M^{3}_{2}-M^{3}_{1}\right)
  +M_{1}M_{2}\bar{F}^{\lambda}(x) S_{1}
  -p_{1}p_{2}\alpha_{1}\alpha_{2}M^{4}_{2}\bar{F}^{\lambda}(x)}
   {\left(p_{1}\alpha_{1}M^{2}_{1}M_{2}
  +p_{2}\alpha_{2}M_{1}M^{2}_{2}\right)^{2}}$,\\
and\\

  $\Frac{\partial r_{V_{2}(p,\alpha)}(x)}  {\partial \alpha_{2}}
  =\Frac{\lambda r(x)p_{1}p_{2}\alpha_{1}M_{1}M_{2}
  \left(M^{3}_{1}-M^{3}_{2}\right)
  +M_{1}M_{2}\bar{F}^{\lambda}(x) S_{2}
  -p_{1}p_{2}\alpha_{1}\alpha_{2}M^{4}_{1}\bar{F}^{\lambda}(x)}
  {\left(p_{1}\alpha_{1}M^{2}_{1}M_{2}
  +p_{2}\alpha_{2}M_{1}M^{2}_{2}\right)^{2}}$,\\
  here $S_{1}=p^{2}_{2}\alpha^{2}_{2}M_{1}M_{2}
  -2p^{2}_{1}\alpha^{2}_{1}M^{2}_{2}\bar{F}^{\lambda}(x)$ and 
  $S_{2}=p^{2}_{1}\alpha^{2}_{1}M_{1}M_{2}
  -2p^{2}_{2}\alpha^{2}_{2}M^{2}_{1}\bar{F}^{\lambda}(x)$.
  
  We have
  \begin{align*}
  	\Frac{\partial r_{V_{2}(p,\alpha)}(x)}  {\partial p_{1}}
  	-\Frac{\partial r_{V_{2}(p,\alpha)}(x)}  {\partial p_{2}}
  	&=\Frac{\lambda r(x)\alpha_{1}\alpha_{2}M_{1}M_{2}
  	 \left(p_{2}M^{3}_{2}-p_{2}M^{3}_{1}-p_{1}M^{3}_{1}+p_{1}M^{3}_{2}\right)}
  	 {\left(p_{1}\alpha_{1}M^{2}_{1}M_{2}
  	+p_{2}\alpha_{2}M_{1}M^{2}_{2}\right)^{2}}\\
  	&=\Frac{\lambda r(x)\alpha_{1}\alpha_{2}M_{1}M_{2}
  	 \left[M^{3}_{2}(p_{1}+p_{2})-M^{3}_{1}(p_{2}+p_{1})\right]}
  	 {\left(p_{1}\alpha_{1}M^{2}_{1}M_{2}
  	+p_{2}\alpha_{2}M_{1}M^{2}_{2}\right)^{2}}\\
  	&=\Frac{\lambda r(x)\alpha_{1}\alpha_{2}M_{1}M_{2}
  	 (p_{1}-p_{2})
  	 \left(M^{3}_{2}-M^{3}_{1}\right)}
  	 {\left(p_{1}\alpha_{1}M^{2}_{1}M_{2}
  	+p_{2}\alpha_{2}M_{1}M^{2}_{2}\right)^{2}},
  \end{align*}
 and 
 
  \begin{align*}
  	\Frac{\partial r_{V_{2}(p,\alpha)}(x)}  {\partial \alpha_{1}}
  	-\frac{\partial r_{V_{2}(p,\alpha)}(x)}  {\partial \alpha_{2}}
  	&=\Frac{p_{1}p_{2}M_{1}M_{2}
  	 \left(\alpha_{2}M^{3}_{2}
  	-\alpha_{2}M^{3}_{1}
  	-\alpha_{1}M^{3}_{1}
  	+\alpha_{1}M^{3}_{2}\right)}
  	 {\left(p_{1}\alpha_{1}M^{2}_{1}M_{2}
  	+p_{2}\alpha_{2}M_{1}M^{2}_{2}\right)^{2}}\\
  	&\quad+\Frac{M_{1}M_{2}\bar{F}^{\lambda}(x)
  	 \left(p^{2}_{2}\alpha^{2}_{2}M_{1}M_{2}
  	-2p^{2}_{1}\alpha^{2}_{1}M^{2}_{2}\bar{F}^{\lambda}(x)
  	-p^{2}_{1}\alpha^{2}_{1}M_{1}M_{2}\right)}
  	 {\left(p_{1}\alpha_{1}M^{2}_{1}M_{2}
  	+p_{2}\alpha_{2}M_{1}M^{2}_{2}\right)^{2}}\\
  	&\quad+\Frac{M_{1}M_{2}\bar{F}^{\lambda}(x)
  	 2p^{2}_{2}\alpha^{2}_{2}M^{2}_{1}\bar{F}^{\lambda}(x)}
  	 {\left(p_{1}\alpha_{1}M^{2}_{1}M_{2}
  	+p_{2}\alpha_{2}M_{1}M^{2}_{2}\right)^{2}}\\
  	&\quad+\Frac{p_{1}p_{2}\alpha_{1}\alpha_{2}M^{4}_{1}\bar{F}^{\lambda}( x)
  	-p_{1}p_{2}\alpha_{1}\alpha_{2}M^{4}_{2}\bar{F}^{\lambda}(x)}
  	 {\left(p_{1}\alpha_{1}M^{2}_{1}M_{2}
  	+p_{2}\alpha_{2}M_{1}M^{2}_{2}\right)^{2}}\\
  	&=\Frac{p_{1}p_{2}M_{1}M_{2} (\alpha_{1}+\alpha_{2})      \left(M^{3}_{2}-M^{3}_{1}\right)}
  	 {\left(p_{1}\alpha_{1}M^{2}_{1}M_{2}
  	+p_{2}\alpha_{2}M_{1}M^{2}_{2}\right)^{2}}\\
  	&\quad+\Frac{M_{1}M_{2}\bar{F}^{\lambda}(x)
  	 \left[M_{1}M_{2}
  	 \left(p^{2}_{2}\alpha^{2}_{2}-p^{2}_{1}\alpha^{2}_{1}\right)
  	+2\bar{F}^{\lambda}(x)
  	 \left(p^{2}_{2}\alpha^{2}_{2}M^{2}_{1}-p^{2}_{1}\alpha^{2}_{1}M^{2}_{2}\right)\right]}
  	 {\left(p_{1}\alpha_{1}M^{2}_{1}M_{2}
  	+p_{2}\alpha_{2}M_{1}M^{2}_{2}\right)^{2}}\\
  	&\quad+\Frac{p_{1}p_{2}\alpha_{1}\alpha_{2}\bar{F}^{\lambda}(x)
  	 \left(M^{4}_{1}-M^ {4}_{2}\right)}
  	 {\left(p_{1}\alpha_{1}M^{2}_{1}M_{2}
  	+p_{2}\alpha_{2}M_{1}M^{2}_{2}\right)^{2}},
  \end{align*}
respectively. Now, upon substituting these expression in (\ref{5}), we obtain
\begin{equation}
  \begin{aligned}
	H(p,\alpha)
	&=\lambda r(x)(p_{1}-p_{2})
	 \frac{\alpha_{1}\alpha_{2}M_{1}M_{2}
	 (p_{1}+p_{2})
	 \left(M^{3}_{2}-M^{3}_{1}\right)}
	 {\left(p_{1}\alpha_{1}M^{2}_{1}M_{2}
	+p_{2}\alpha_{2}M_{1}M^{2}_{2}\right)^{2}}\\
	&\quad+\lambda r(x)(\alpha_{1}-\alpha_{2})
	 \frac{p_{1}p_{2}M_{1}M_{2} (\alpha_{1}+\alpha_{2}) \left(M^{3}_{2}-M^{3}_{1}\right)}
	 {\left(p_{1}\alpha_{1}M^{2}_{1}M_{2}
	+p_{2}\alpha_{2}M_{1}M^{2}_{2}\right)^{2}}\\
	&\quad+\lambda r(x)(\alpha_{1}-\alpha_{2})
	 \frac{M_{1}M_{2}\bar{F}^{\lambda}(x)
	 \left[M_{1}M_{2}
	 \left(p^{2}_{2}\alpha^{2}_{2}-p^{2}_{1}\alpha^{2}_{1}\right)
	+2\bar{F}^{\lambda}(x)
	 \left(p^{2}_{2}\alpha^{2}_{2}M^{2}_{1}-p^{2}_{1}\alpha^{2}_{1}M^{2}_{2}\right)\right]}
	 {\left(p_{1}\alpha_{1}M^{2}_{1}M_{2}
	+p_{2}\alpha_{2}M_{1}M^{2}_{2}\right)^{2}}\\
	&\quad+\lambda r(x)(\alpha_{1}-\alpha_{2})
	 \frac{p_{1}p_{2}\alpha_{1}\alpha_{2}\bar{F}^{\lambda}(x)
	 \left(M^{4}_{1}-M^{4}_{2}\right)}
	 {\left(p_{1}\alpha_{1}M^{2}_{1}M_{2}
	+p_{2}\alpha_{2}M_{1}M^{2}_{2}\right)^{2}}\\
	&\overset{\rm{sgn}}=\lambda r(x)
     (p_{1}-p_{2}) (p_{1}+p_{2})
     \alpha_{1}\alpha_{2}M_{1}M_{2}
     \left(M^{3}_{2}-M^{3}_{1}\right)\\
    &\quad+\lambda r(x)(\alpha_{1}-\alpha_{2})  (\alpha_{1}+\alpha_{2})
     p_{1}p_{2}M_{1}M_{2}
     \left(M^{3}_{2}-M^{3}_{1}\right)\\
    &\quad+\lambda r(x)(\alpha_{1}-\alpha_{2})
     M_{1}M_{2}\bar{F}^{\lambda}(x)
     \left[M_{1}M_{2}
     \left(p^{2}_{2}\alpha^{2}_{2}-p^{2}_{1}\alpha^{2}_{1}\right)
    +2\bar{F}^{\lambda}(x)
     \left(p^{2}_{2}\alpha^{2}_{2}M^{2}_{1}-p^{2}_{1}\alpha^{2}_{1}M^{2}_{2}\right)\right]\\
    &\quad+\lambda r(x)
     (\alpha_{1}-\alpha_{2})
     p_{1}p_{2}\alpha_{1}\alpha_{2}\bar{F}^{\lambda}(x)
     \left(M^{4}_{1}-M^{4}_{2}\right)\\
    &\overset{\rm{sgn}}=(p_{1}-p_{2}) (p_{1}+p_{2})
     \alpha_{1}\alpha_{2}M_{1}M_{2}
     \left(M^{3}_{2}-M^{3}_{1}\right)\\
    &\quad+(\alpha_{1}-\alpha_{2})  (\alpha_{1}+\alpha_{2})
     p_{1}p_{2}M_{1}M_{2}
     \left(M^{3}_{2}-M^{3}_{1}\right)\\
    &\quad+(\alpha_{1}-\alpha_{2})
     M_{1}M_{2}\bar{F}^{\lambda}(x)
     \left[M_{1}M_{2}
     \left(p^{2}_{2}\alpha^{2}_{2}-p^{2}_{1}\alpha^{2}_{1}\right)
    +2\bar{F}^{\lambda}(x)
     \left(p^{2}_{2}\alpha^{2}_{2}M^{2}_{1}-p^{2}_{1}\alpha^{2}_{1}M^{2}_{2}\right)\right]\\
    &\quad+(\alpha_{1}-\alpha_{2})
     p_{1}p_{2}\alpha_{1}\alpha_{2}\bar{F}^{\lambda}(x)
     \left(M^{4}_{1}-M^{4}_{2}\right). 
  \end{aligned}
  \label{6}
\end{equation}

    The assumption that $(p,\alpha)\in\mathcal{K}_{2}$ implies that \begin{center}
    	$(p_{1}-p_{2})(\alpha_{1}-\alpha_{2})\le0$,
    \end{center} 
    which means that $p_{1}\ge p_{2}$ and $\alpha_{1}\le \alpha_{2}$, 
    (or $p_{1}\le p_{2}$ and $\alpha_{1}\ge \alpha_{2}$).
    We present the proof only for the case when $p_{1}\ge p_{2}$ and $\alpha_{1}\le \alpha_{2}$, 
    since the proof for the other case is quite similar. Because $\alpha_{1}\le \alpha_{2}$, and $p_{1}\alpha_{1}=p_{2}\alpha_{2}$, 
    we have 
    \begin{center}
    	$M_{1}\le M_{2},$
    \end{center}
     and 
    \begin{center}
    	$p^{2}_{1}\alpha^{2}_{1}M^{2}_{1} \le p^{2}_{2}\alpha^{2}_{2}M^{2}_{2}$.
    \end{center}
    
    We readily observe that equation (\ref{6}) is positive for both the third and fourth parts. Moreover, according to the assumption that 
    \begin{center}
    	$p_{1}\ge p_{2}$,
    \end{center} 
    the first part of equation (\ref{6}) is positive and the second part of equation (\ref{6}) is non-positive. We observe that the first part is greater than the second part through comparison. Upon combining these observations, we have \begin{center}
    	$H(p,\alpha)\ge 0$.
    \end{center}
    
     Condition (ii) of Lemma \ref{lem-1} is also satisfied. The proof of Theorem \ref{th-5} is completed. 
  
\end{proof}
    Example \ref{e-6} provides an illustration of  Theorem \ref{th-5}.
\begin{example}\label{e-6} 
	Consider the finite mixture $V_{2}(p,\alpha)$ and $W_{2}(q,\beta)$, where $(p,\alpha)\in\mathcal{K}_{2}$,
	$\lambda=0.2$,
    $(p_{1},p_{2})=(0.3,0.7)$,
	$(\alpha_{1},\alpha_{2})=(0.7,0.3)$,
	$(q_{1},q_{2})=(0.34,0.66)$,
	and
	$(\beta_{1}\beta_{2})=(0.66,0.34)$.
	It can easily observe that $(q,\beta)=(p,\alpha)T_{0.9}$,
	implies $(p,\alpha)\gg(q,\beta)$.
	We take $\bar{F}(x)=exp\left\{-ax\right\}$ with $a=3$.
	The survival functions of the two finite mixture models are plotted in Figure \ref{F-6} . It can be seen in Figure \ref{F-6} that $V_{2}(p,\alpha)\ge_{hr}W_{2}(q,\beta)$.
	Therefore, the validity of Theorem \ref{th-5} is verified.
	\begin{figure}[bthp]\label{F-6}
		\centering
		\includegraphics[width=0.7\linewidth]{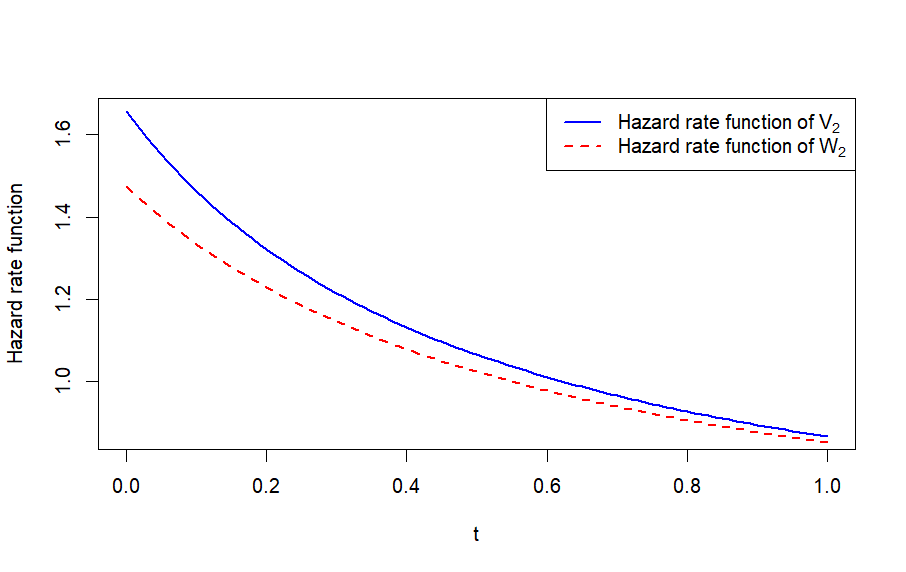}
		\caption{Plots of $r_{V_{2}(p,\alpha)}(x)$ and $r_{W_{2}(q,\beta)}(x)$ for $x=t/(1-t)$, $t\in[0,1]$.}
		\label{fig:}
	\end{figure}

\end{example}

Similarly, based on the usual stochastic order, we can get the result of the following hazard rate orders.

Theorem \ref{th-6}  shows how to compare the hazard rate orders of finite mixture models for two $n\times n$-order matrices under chain majorization.
\begin{theorem}\label{th-6} 
	Let $\bar{F}_{V_{n}{(p,\alpha)}}(x)=\sum_{i=1}^{n} p_{i} \frac{\alpha _{i} {\bar{F}^{\lambda }(x) } }{1-\bar{\alpha_{i} }\bar{F}^{\lambda}(x)}$
	and $\bar{F}_{W_{n}{(q,\beta)}}(x)=\sum_{i=1}^{n} q_{i} \frac{\beta_{i}{\bar{F}^{\lambda}(x)}}{1-\bar{\beta_{i}}\bar{F}^{\lambda}(x)}$
	be the survival functions of the finite mixtures with modified proportional hazard rates model corresponding to $V_{n}(p,\alpha)$ and $W_{n}(q,\beta)$, Further suppose $p_{i}\alpha_{i}=p_{j}\alpha_{j}$ and $r(x)$ is positive, then for $(p,\alpha)\in\mathcal{K}_{n}$,  respectively.
	 \begin{center}
	 	$\begin{pmatrix}
	 		q_{1}&\cdots&q_{n}\\
	 		\beta_{1}&\cdots&\beta_{n}
	     \end{pmatrix}
	     =
	     \begin{pmatrix}
	     	p_{1}&\cdots&p_{n}\\
	     	\alpha_{1}&\cdots&\alpha_{n}
	     \end{pmatrix}T
	     \Longrightarrow
	     V_{n}(p,\alpha)\ge_{hr}W_{n}(q,\beta).$
	 \end{center}
\end{theorem}	

Corollary \ref{c-5}  shows how to compare the hazard rate orders of finite mixture models for two $n\times n$-order matrices under chain majorization. In the process of chain optimization, the T-transform matrices exhibits the same structure.
\begin{corollary}\label{c-5}
   	Let $\bar{F}_{V_{n}{(p,\alpha)}}(x)=\sum_{i=1}^{n} p_{i} \frac{\alpha _{i} {\bar{F}^{\lambda }(x) } }{1-\bar{\alpha_{i} }\bar{F}^{\lambda}(x)}$
    and $\bar{F}_{W_{n}{(q,\beta)}}(x)=\sum_{i=1}^{n} q_{i} \frac{\beta_{i}{\bar{F}^{\lambda}(x)}}{1-\bar{\beta_{i}}\bar{F}^{\lambda}(x)}$
    be the survival functions of the finite mixtures with modified proportional hazard rates model corresponding to $V_{n}(p,\alpha)$ and $W_{n}(q,\beta)$, Further suppose $p_{i}\alpha_{i}=p_{j}\alpha_{j}$ and $r(x)$ is positive, T-transform matrices $T_{1},\cdots,T_{n}$ have the same structure, then for $(p,\alpha)\in\mathcal{K}_{n}$, respectively.
    \begin{center}
    	$\begin{pmatrix}
    		q_{1}&\cdots&q_{n}\\
    		\beta_{1}&\cdots&\beta_{n}
    	\end{pmatrix}
    	=
    	\begin{pmatrix}
    		p_{1}&\cdots&p_{n}\\
    		\alpha_{1}&\cdots&\alpha_{n}
    	\end{pmatrix}
    	T_{1}\cdots T_{k}
    	\Longrightarrow
    	V_{n}(p,\alpha)\ge_{hr}W_{n}(q,\beta).$
    \end{center}
\end{corollary}

Corollary \ref{c-5}  shows how to compare the hazard rate orders of finite mixture models for two $n\times n$-order matrices under chain majorization. In the process of chain optimization, the T-transform matrices exhibits the different structure.
\begin{corollary} \label{c-6}
	Let $\bar{F}_{V_{n}{(p,\alpha)}}(x)=\sum_{i=1}^{n} p_{i} \frac{\alpha _{i} {\bar{F}^{\lambda }(x) } }{1-\bar{\alpha_{i} }\bar{F}^{\lambda}(x)}$
	and $\bar{F}_{W_{n}{(q,\beta)}}(x)=\sum_{i=1}^{n} q_{i} \frac{\beta_{i}{\bar{F}^{\lambda}(x)}}{1-\bar{\beta_{i}}\bar{F}^{\lambda}(x)}$
	be the survival functions of the finite mixtures with modified proportional hazard rates model corresponding to $V_{n}(p,\alpha)$ and $W_{n}(q,\beta)$, Further suppose $p_{i}\alpha_{i}=p_{j}\alpha_{j}$ and $r(x)$ is positive, T-transform matrices have the different structure, then for $(p,\alpha)\in\mathcal{K}_{n}$ and $(p,\alpha)T_{1}\cdots T_{i}\in\mathcal{K}_{n}$ $(i=1,\cdots,k-1$, where $k\ge2)$, respectively.
	\begin{center}
		$\begin{pmatrix}
			q_{1}&\cdots&q_{n}\\
			\beta_{1}&\cdots&\beta_{n}
		\end{pmatrix}
		=
		\begin{pmatrix}
			p_{1}&\cdots&p_{n}\\
			\alpha_{1}&\cdots&\alpha_{n}
		\end{pmatrix}
		T_{1}\cdots T_{k}
		\Longrightarrow
		V_{n}(p,\alpha)\ge_{hr}W_{n}(q,\beta).$
	\end{center}
\end{corollary}
\begin{remark} 
	In this section, we discuss the finite mixture model with modified proportional hazard rate model in the sense of hazard rate order. Relative to \cite{panja2022stochastic}, we use chain majorization to address the hazard rate order. In our work, we generalize some results of the hazard rate order of two-component mixture models in to $n(>2)$ component mixture model, and the $n$ parameters are not the same.
\end{remark}
\section{Star order and Lorenz order of the finite mixtures}\label{sec-5}
     We consider finite mixtures of $n(=n_1+n_2)$ components where
     $n_1$ components are drawn from a particular homogeneous subgroups and 
     rest $n_2$ components are drawn from anther homogeneous subgroup. 
     Such that $\alpha_{i}=\alpha_{1}$
     for $i=1,2,\cdots,n_{1}$ and $\alpha_{i}=\alpha_{2}$ 
     for $i=n_{1}+1,\cdots,n$. 
     Let $M_{n}(p,\alpha)$ be the random variable representing the finite mixtures.
     And $p_{i}$ are the corresponding mixture proportion such that $p_{i}=p_{1}$ 
     for $i=1,2,\cdots,n_{1}$ and $p_{i}=p_{2}$ 
     for $i=n_{1}+1,\cdots,n$, 
     so that $n_{1}p_{1}+n_{2}p_{2}=1$. 
     Here we use the notations
     $\alpha=(\underbrace{\alpha_{1},\alpha_{1},\cdots,\alpha_{1}}_{n_{1}\ terms},\underbrace{\alpha_{2},\alpha_{2},\cdots,\alpha_{2}}_{n_{2}\ terms})$ 
     and
     $p=(\underbrace{p_{1},p_{1},\cdots,p_{1}}_{n_{1}\ terms},\underbrace{p_{2},p_{2},\cdots,p_{2}}_{n_{2}\ terms})$.
     The following Theorem \ref{th-7} and Corollary \ref{c-7} show that two random variables are compared by weakly supmajorization under star order and Lorenz order when the mixture proportions $p_{i}$ and tilt parameters $\alpha_{i}$ are different.
\begin{theorem}\label{th-7} 
	Let $M_{n}(p,\alpha)$ and $N_{n}(p,\beta)$ be two finite mixtures with same mixture proportion and $\frac{1-\bar{F}^{\lambda}(x)}{x\lambda r(x)}$
	is decreasing in $x$. 
	If $\alpha_{1}\ge \beta_{1}\ge\beta_{2}\ge\alpha_{2}$, $p_{1}\le p_{2}$
	and $\boldsymbol{\alpha} \overset{w}\preceq\boldsymbol{\beta} $,
	then $M_{n}(p,\alpha)\ge_{\star}N_{n}(p,\beta)$.
\end{theorem}
     
\begin{proof} 
	We have the distribution function of $M_{n}(p,\alpha)$ as
	\begin{center}
		 $F_{M_{n}(p,\alpha)}(x)
	=n_{1}p_{1}F_{\alpha_{1}}+n_{2}p_{2}F_{\alpha_{2}}$,
	\end{center}
	here
	\begin{center}
	$F_{\alpha_{i}}=\frac{1-\bar{F}^{\lambda}(x)}{1-\bar{\alpha_{i}}\bar{F}^{\lambda}(x)}, i=1,2$.
	\end{center}
	
    We prove the theorem using the concept of Lemma \ref{lem-3}. 
    Under the conditions 
    $\alpha_{1}\ge \beta_{1}\ge \beta_{2}\ge \alpha_{2}$ 
    and $\boldsymbol{\alpha} \overset{w}\preceq\boldsymbol{\beta} $, 
    we have $n_{1}\alpha_{1}+n_{2}\alpha_{2}\ge n_{1}\beta_{1}+n_{2}\beta_{2}$.
	
	Case I: $n_{1}\alpha_{1}+n_{2}\alpha_{2}= n_{1}\beta_{1}+n_{2}\beta_{2}$. 
	Without loss of generality we consider 
	$n_{1}\alpha_{1}+n_{2}\alpha_{2}= n_{1}\beta_{1}+n_{2}\beta_{2}=1$. 
	Let $\alpha_{1}=\alpha$ and $\beta_{1}=\beta$, 
	where $\alpha,\beta\in\left [ 1/(n_{1}+n_{2}),1/n_{1}\right)$,
    and we write $F_{M_{n}(p,\alpha)}(x)=F_{M_{\alpha}(x)}$, just to indicate that it becomes an expression of $\alpha$. Now\\
    
    $\Frac{\partial F_{M_{\alpha}}(x)}{\partial \alpha}
       =-\bar{F}^{\lambda}(1-\bar{F}^{\lambda}(x))
       \left[\Frac{n_{1}p_{1}}
             {\left(1-\bar{\alpha}\bar{F}^{\lambda}(x)\right)^2 }
             -
             \Frac{n_{1}p_{2}}
             {\left(1-
             	\left(1-\frac{1-n_{1}\alpha}  {n_{2}}\right)
             \bar{F}^{\lambda}(x))^2\right)}
        \right]$.\\
        
   Further,\\
   
   $f_{M_{\alpha}}(x)=
     \lambda f(x)\bar{F}^{\lambda-1}(x)
     \left[\Frac{n_{1}p_{1}\alpha}
                {\left(1-\bar{\alpha}\bar{F}^{\lambda}(x)\right)^2}
      +\Frac{p_{2}(1-n_{1}\alpha)}
            {\left(1-
            	\left(1-\frac{1-n_{1}\alpha}{n_{2}}\right)
            \bar{F}^{\lambda}(x)\right)^2}
     \right]$.\\

   We have
 
\begin{equation}
\begin{aligned}
 \frac{\partial F_{M_{\alpha}}/\partial\alpha}{xf_{M_{\alpha}}(x)}
 &=-\frac{1-\bar{F}^{\lambda}(x)}{x\lambda r(x)}
   \left[\frac{n_{1}p_{1}
               \left(1-
                     \left(1-\frac{1-n_{1}\alpha}{n_{2}}\right)\bar{F}^{\lambda}(x)\right)^2
               -
               n_{1}p_{2}
               \left(1-\bar{\alpha}\bar{F}^{\lambda}(x)\right)^2}
               {n_{1}p_{1}\alpha
               \left(\left(1-
                     \left(1-\frac{1-n_{1}\alpha}{n_{2}}\right)\bar{F}^{\lambda}(x)\right)^2\right)
               +
               p_{2}(1-n_{1}\alpha)
               \left(1-\bar{\alpha}\bar{F}^{\lambda}(x)\right)^2}
   \right]\\
 &=-\frac{1-\bar{F}^{\lambda}(x)}{x\lambda r(x)}
   \left[\frac{n_{1}p_{1}\alpha
   	          \left(\left(1-
   	               \left(1-\frac{1-n_{1}\alpha}{n_{2}}\right)\bar{F}^{\lambda}(x)\right)^2\right)
   	          +
   	          p_{2}(1-n_{1}\alpha)
   	          \left(1-\bar{\alpha}\bar{F}^{\lambda}(x)\right)^2}
   	          {n_{1}p_{1}
 	          \left(1-
 	              \left(1-\frac{1-n_{1}\alpha}{n_{2}}\right)\bar{F}^{\lambda}(x)\right)^2
 	         -
   	         n_{1}p_{2}
 	         \left(1-\bar{\alpha}\bar{F}^{\lambda}(x)\right)^2}
   \right]^{-1}\\
 &=-\frac{1-\bar{F}^{\lambda}(x)}{x\lambda r(x)}
         \left[\alpha+
         \frac{p_{2}
         	\left(1-\bar{\alpha}\bar{F}^{\lambda}(x)\right)^2}
            {n_{1}p_{1}
         	\left(1-
         	\left(1-\frac{1-n_{1}\alpha}{n_{2}}\right)\bar{F}^{\lambda}(x)\right)^2
         	-
         	n_{1}p_{2}
         	\left(1-\bar{\alpha}\bar{F}^{\lambda}(x)\right)^2}
         \right]^{-1}\\
  &=-\frac{1-\bar{F}^{\lambda}(x)}{x\lambda r(x)}
         \left[\alpha+
              \left(\frac{p_{2}}{n_{1}}\right)
              \left(\frac{p_{1}
                          \left(1-
                               \left(1-
                                    \frac{1-n_{1}\alpha}{n_{2}}
                               \right)
                          \bar{F}^{\lambda}(x)
                          \right)^2}
                         {1-\bar{\alpha}\bar{F}^{\lambda}(x)}
              -p_{2}
              \right)^{-1}
         \right]^{-1}\\
   &=-\frac{1-\bar{F}^{\lambda}(x)}{x\lambda r(x)}
           \bigtriangleup (x).
\end{aligned}
\label{7}
\end{equation}
  
  For $\alpha=\alpha_{1}\ge\alpha_{2}=\left(\frac{1-n_{1}\alpha}{n_{2}}\right)$, we have
 \begin{center}
 $\left(1-\bar{\alpha}\bar{F}^{\lambda}(x)\right)^2
  \ge
  \left(1-\left(1-\frac{1-n_{1}\alpha}{n_{2}}
          \right)\bar{F}^{\lambda}(x)
  \right)^2$,
 \end{center} 
 and
 \begin{center}
 	$\frac{\left(1-\left(1-\frac{1-n_{1}\alpha}{n_{2}}
 		\right)\bar{F}^{\lambda}(x)
 		\right)}
 	{1-\bar{\alpha}\bar{F}^{\lambda}(x)}$
 	is increasing in $x$.
 \end{center} 
  so, for $p_{1}\le p_{2}$, from the first equation of (\ref{7}), we obtain 
  \begin{center}
  	$\bigtriangleup (x)\le0$. 
  \end{center}  
  
  So that from (\ref{7})
  it following that 
  \begin{center}
  	$\bigtriangleup (x)$ is also increasing in $x$.
  \end{center}
 
  Further,
   \begin{center}
   	$\frac{1-\bar{F}^{\lambda}(x)}{x\lambda r(x)}
      \bigtriangleup (x)$ is decreasing in $x$. 
   \end{center}
   
  Hence, 
  \begin{center}
  	$\frac{\partial F_{M_{\alpha}}(x)/\partial\alpha}   {f_{M_{\alpha}}(x)}$ is decreasing in $x$.
  \end{center}
  
  Case II: Suppose $n_{1}\alpha_{1}+n_{2}\alpha_{2}\ge n_{1}\beta_{1}+n_{2}\beta_{2}$. In this case there must exist some $\alpha^{'}_{1}$ such that $\alpha_{1}\ge \alpha^{'}_{1}\ge \beta_{1}$
  such that $n_{1}\alpha^{'}_{1}+n_{2}\alpha_{2}=n_{1}\beta_{1}+n_{2}\beta_{2}$.
  Let us denoted $\boldsymbol{\alpha}^{'}=(\underbrace{\alpha^{'}_{1},\alpha^{'}_{1},\cdots,\alpha^{'}_{1}}_{n_{1}terms},\underbrace{\alpha_{2},\alpha_{2},\cdots,\alpha_{2}}_{n_{2}terms})$.
  Then from previous case we have $M_{n}(p,\alpha^{'})\ge_{\star} M_{n}(p,\beta)$. 
  
  Now we need to show that $M_{n}(p,\alpha)\ge_{\star}N_{n}(p,\alpha^{'})$.
  Let $\eta=\alpha_{1}-\alpha_{2}$, $\eta^{'}=\alpha^{'}_{1}-\alpha_{2}$,
  and $\eta^{\star}=\alpha_{1}-\eta$.
  then $\eta\ge\eta^{'}\ge0$.
  According to Lemma \ref{lem-3}, it is sufficient to show that\\
  \begin{align*}
  \frac{\partial F_{\eta}/\partial\eta}{xf_{\eta}(x)}
  &=\frac{1-\bar{F}^{\lambda}(x)}{\lambda r(x)}
    \times
    \left[
    \frac{n_{2}p_{2}
  	     \left(1-\bar{\alpha}_{1}\bar{F}^{\lambda}(x)
  	     \right)^2}
         {\alpha_{1}n_{1}p_{1}
         \left(1-\bar{\eta^*}\bar{F}^{\lambda}(x)
         \right)^2
         +n_{2}p_{2}
         \left(1-\bar{\alpha}_{1}\bar{F}^{\lambda}(x)
         \right)^2}
    \right]\\
  &=\frac{1-\bar{F}^{\lambda}(x)}{\lambda r(x)}
    \times
    \left[\eta^*+
          \frac{n_{1}p_{1}\alpha_{1}}
               {n_{2}p_{2}}
          \left(\frac{1-\bar{\eta^*}\bar{F}^{\lambda}(x)}
                     {1-\bar{\alpha}_{1}\bar{F}^{\lambda}(x)}
          \right)^2
    \right]^{-1}\\
  &=\frac{1-\bar{F}^{\lambda}(x)}{\lambda r(x)}
    \times
    \bigtriangledown(x)
  \end{align*}
  is decreasing in $x$. 
  
  As $\alpha_{1}\ge\eta^*$,
  \begin{center}
  	$\Frac{\rm{d}}{\mathrm{d}x}
   \left(\Frac{1-\bar{\eta^*}\bar{F}^{\lambda}(x)}
              {1-\bar{\alpha}_{1}\bar{F}^{\lambda}(x)}
   \right)
   =\Frac{\lambda\bar{F}^{\lambda-1}(x)f(x)
   	(\bar{\eta}^*-\bar{\alpha}_{1})}
   {\left(1-\bar{\alpha}_{1}\bar{F}^{\lambda}(x)
   	\right)^2}
   \ge0.$ 
  \end{center}
   
   So,
   \begin{center}
   	 $\bigtriangledown(x)$ is non-negative and decreasing in $x$.
   \end{center}
   
   Consequently,
    \begin{center}
    $\Frac{\partial F_{\eta}/\partial\eta}{xf_{\eta}(x)}$
   is decreasing in $x$.
    \end{center} 
    
    The proof of the Theorem \ref{th-7} is completed.
\end{proof} 

Example \ref{e-7} provides an illustration of Theorem \ref{th-7}.
\begin{example}\label{e-7}
	Consider the finite mixture $M_{2}(p,\alpha)$ and $N_{2}(q,\beta)$ with 
	$\lambda=0.1$,
	$(p_{1},p_{2})=(0.3,0.05)$, $(\alpha_{1},\alpha_{1},\alpha_{1},\alpha_{2},\alpha_{2})=(8,8,8,2,2)$, 
	$(\beta_{1},\beta_{1},\beta_{1},\beta_{2},\beta_{2})=(6,6,6,3,3)$, $n_{1}=3$ and $n_{2}=2$. 
	We take the survival function $\bar{F}(x)=(\frac{1}{1+x^a})^b$ with $a=0.2,b=0.5$.
	Suppose $S(x)=F^{-1}_{2}(F_{1}(x))$.
	\begin{figure}[htbp]\label{F-7}
		\centering
		\includegraphics[width=0.8\linewidth]{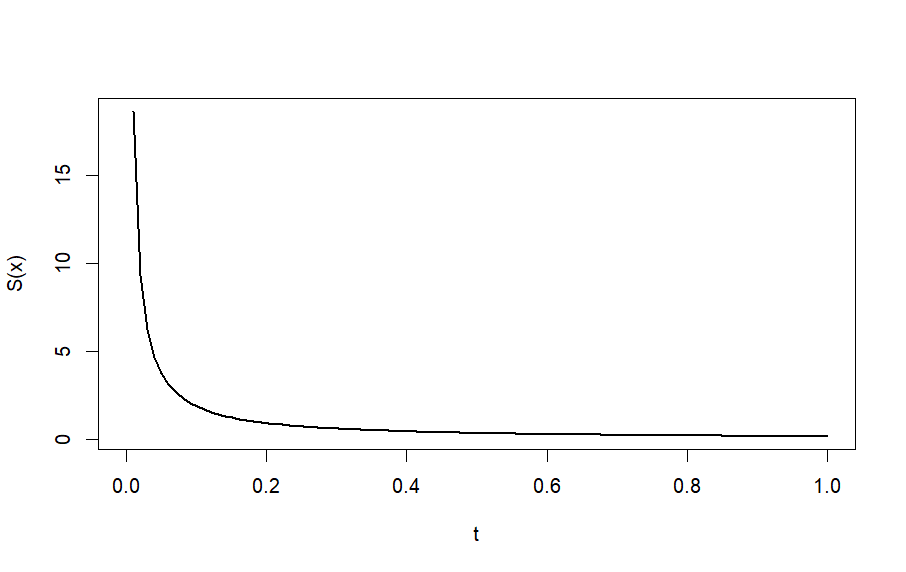}
		\caption{Plot of derivative of $S(x)$ for $x=t/(1-t)$, $t\in[0,1]$.}
		\label{fig:th7}
	\end{figure}
	It is easy to observe that all conditions of Theorem \ref{th-7} are satisfied. From Figure \ref{F-7}, $F^{-1}_{2}(F_{1}(x))$ is decreasing in $x$, which implies that $M_{5}(p,\alpha)\ge_{\star}N_{5}(q,\beta)$.
	Therefore, the validity of Theorem \ref{th-7} has been confirmed.
\end{example}   
   
\begin{corollary} \label{c-7} 
	Under the same setup of Theorem \ref{th-7}, if $\alpha_{1}\ge \beta_{1}\ge\beta_{2}\ge\alpha_{2}$, $p_{1}\le p_{2}$
	and $\boldsymbol{\alpha} \overset{w}\preceq\boldsymbol{\beta} $
	then $M_{n}(p,\alpha)\ge_{Lorenz}N_{n}(p,\beta)$.
\end{corollary}

\begin{remark}
	In this section, we discuss the star order and Lorenz order of the finite mixture with modified proportional hazard rate model by weakly supmajorization order. \cite{panja2022stochastic} showed the finite mixture with PO, PHR and PRH in the sense of star order and Lorenz order. We discuss correlation orders of the finite mixture with MPHR through \cite{panja2022stochastic}. 
	
\end{remark}

\section{Concluding remarks}\label{sec-6}

   In this paper, we have considered two finite mixture with modified proportional hazard rate model for modeling heterogeneity in the population. Under certain conditions on the model parameters, we have  presented some sufficient conditions for the comparison of stochastic comparisons of two finite mixture models in the sense of the usual stochastic order, hazard rate order, star order and Lorenz order. The sufficient conditions are mainly dependent on the concept of majorization orders, involving chain majorization order, majorization order and the weak supermajorization order. Several examples have also been presented for illustrating for the established results.
   
   There are possibilities to extend the results of this paper. It will be worth to extend the results of hazard rate, and star ordering results for n component heterogeneous mixture models where the corresponding random variables follow modified proportional reversed hazard rate model. we have established in this work stochastic orderings for finite mixtures of modified proportional hazard rate models when either the mixture proportions and tilt parameters are different, or the mixture proportions and modified proportional hazard rates. It will naturally be of interest to consider the general situation when all  mixture proportions, tilt parameters and modified proportional hazard rates are different. We are currently looking into this problem and hope to report the findings in a future paper.

\section*{Acknowledgements}
This work is supported by the National Natural Science Foundation of China [grant number 12361060], College Teachers Innovation
Foundation Project of Gansu Provincial Education Department [grant number 2024A-002].

\appendix

\bibliographystyle{plainnat}
\setlength{\bibsep}{0.8ex}

\begin{spacing}{1.0}
	\bibliography{references}

\begin{thebibliography}{30}
\providecommand{\natexlab}[1]{#1}
\providecommand{\url}[1]{\texttt{#1}}
\expandafter\ifx\csname urlstyle\endcsname\relax
  \providecommand{\doi}[1]{doi: #1}\else
  \providecommand{\doi}{doi: \begingroup \urlstyle{rm}\Url}\fi

\bibitem[Albabtain et~al.(2020)Albabtain, Shrahili, Al-Shehri, and
  Kayid]{albabtain2020stochastic}
Abdulhakim~A Albabtain, Mansour Shrahili, Mashael~A Al-Shehri, and Mohamed
  Kayid.
\newblock Stochastic comparisons of weighted distributions and their mixtures.
\newblock \emph{Entropy}, 22\penalty0 (8):\penalty0 843, 2020.

\bibitem[Amini-Seresht and Zhang(2017)]{amini2017stochastic}
Ebrahim Amini-Seresht and Yiying Zhang.
\newblock Stochastic comparisons on two finite mixture models.
\newblock \emph{Operations Research Letters}, 45\penalty0 (5):\penalty0
  475--480, 2017.

\bibitem[Asadi et~al.(2019)Asadi, Ebrahimi, and Soofi]{asadi2019alpha}
Majid Asadi, Nader Ebrahimi, and Ehsan~S Soofi.
\newblock The alpha-mixture of survival functions.
\newblock \emph{Journal of Applied Probability}, 56\penalty0 (4):\penalty0
  1151--1167, 2019.

\bibitem[Balakrishnan et~al.(2014)Balakrishnan, Haidari, and
  Masoumifard]{balakrishnan2014stochastic}
Narayanaswamy Balakrishnan, Abedin Haidari, and Khaled Masoumifard.
\newblock Stochastic comparisons of series and parallel systems with
  generalized exponential components.
\newblock \emph{IEEE Transactions on Reliability}, 64\penalty0 (1):\penalty0
  333--348, 2014.

\bibitem[Balakrishnan et~al.(2018)Balakrishnan, Barmalzan, and
  Haidari]{balakrishnan2018modified}
Narayanaswamy Balakrishnan, Ghobad Barmalzan, and Abedin Haidari.
\newblock Modified proportional hazard rates and proportional reversed hazard
  rates models via marshall-olkin distribution and some stochastic comparisons.
\newblock \emph{Journal of the Korean Statistical Society}, 47:\penalty0
  127--138, 2018.

\bibitem[Bansal and Gupta(2019)]{bansal2019stochastic}
Shilpa Bansal and Nitin Gupta.
\newblock Stochastic comparisons of multivariate reversed frailty model.
\newblock \emph{Statistica Neerlandica}, 73\penalty0 (4):\penalty0 496--513,
  2019.

\bibitem[Barmalzan et~al.(2021)Barmalzan, Kosari, and
  Zhang]{barmalzan2021stochastic}
Ghobad Barmalzan, Sajad Kosari, and Yiying Zhang.
\newblock On stochastic comparisons of finite $\alpha$-mixture models.
\newblock \emph{Statistics \& Probability Letters}, 173:\penalty0 109083, 2021.

\bibitem[Barmalzan et~al.(2022{\natexlab{a}})Barmalzan, Hosseinzadeh, and
  Balakrishnan]{barmalzan2022dispersion}
Ghobad Barmalzan, Ali~Akbar Hosseinzadeh, and Narayanaswamy Balakrishnan.
\newblock Dispersion and variability orders of mixture exponential
  distributions and their sample spacings, and some associated
  characterizations.
\newblock \emph{Communications in Statistics-Theory and Methods}, 51\penalty0
  (24):\penalty0 8657--8670, 2022{\natexlab{a}}.

\bibitem[Barmalzan et~al.(2022{\natexlab{b}})Barmalzan, Kosari, and
  Balakrishnan]{barmalzan2022orderings}
Ghobad Barmalzan, Sajad Kosari, and Narayanaswamy Balakrishnan.
\newblock Orderings of finite mixture models with location-scale distributed
  components.
\newblock \emph{Probability in the Engineering and Informational Sciences},
  36\penalty0 (2):\penalty0 461--481, 2022{\natexlab{b}}.

\bibitem[Bhakta et~al.(2023)Bhakta, Kundu, and Kayal]{bhakta2023stochastic}
Raju Bhakta, Pradip Kundu, and Suchandan Kayal.
\newblock Stochastic orderings between two finite mixture models with
  inverted-kumaraswamy distributed components.
\newblock \emph{arXiv preprint arXiv:2311.17568}, 2023.

\bibitem[Bhakta et~al.(2024)Bhakta, Kayal, and
  Balakrishnan]{bhakta2024ordering}
Raju Bhakta, Suchandan Kayal, and Narayanaswamy Balakrishnan.
\newblock Ordering results between two multiple-outlier finite
  $\delta$-mixtures.
\newblock \emph{Statistics \& Probability Letters}, page 110193, 2024.

\bibitem[Cha and Finkelstein(2013)]{cha2013failure}
Ji~Hwan Cha and Maxim Finkelstein.
\newblock The failure rate dynamics in heterogeneous populations.
\newblock \emph{Reliability Engineering \& System Safety}, 112:\penalty0
  120--128, 2013.

\bibitem[Fang et~al.(2023)Fang, Balakrishnan, Huang, and Zhang]{fang2023usual}
Longxiang Fang, Narayanaswamy Balakrishnan, Wenyu Huang, and Shuai Zhang.
\newblock Usual stochastic ordering of the sample maxima from dependent
  distribution-free random variables.
\newblock \emph{Statistica Neerlandica}, 77\penalty0 (1):\penalty0 99--112,
  2023.

\bibitem[Finkelstein(2008)]{finkelstein2008failure}
Maxim Finkelstein.
\newblock \emph{Failure rate modelling for reliability and risk}.
\newblock Springer Science \& Business Media, 2008.

\bibitem[Ghitany(2005)]{ghitany2005marshall}
ME~Ghitany.
\newblock Marshall-olkin extended pareto distribution and its application.
\newblock \emph{International Journal of Applied Mathematics}, 18\penalty0
  (1):\penalty0 17, 2005.

\bibitem[Ghitany et~al.(2007)Ghitany, Al-Awadhi, and
  Alkhalfan]{ghitany2007marshall}
ME~Ghitany, FA~Al-Awadhi, and LA~Alkhalfan.
\newblock Marshall--olkin extended lomax distribution and its application to
  censored data.
\newblock \emph{Communications in Statistics—Theory and Methods}, 36\penalty0
  (10):\penalty0 1855--1866, 2007.

\bibitem[Hazra and Finkelstein(2018)]{hazra2018stochastic}
Nil~Kamal Hazra and Maxim Finkelstein.
\newblock On stochastic comparisons of finite mixtures for some semiparametric
  families of distributions.
\newblock \emph{Test}, 27\penalty0 (4):\penalty0 988--1006, 2018.

\bibitem[Marshall and Olkin(2007)]{marshall2007life}
Albert~W Marshall and Ingram Olkin.
\newblock \emph{Life distributions}, volume~13.
\newblock Springer, 2007.

\bibitem[Marshall et~al.(1979)Marshall, Olkin, and
  Arnold]{marshall1979inequalities}
Albert~W Marshall, Ingram Olkin, and Barry~C Arnold.
\newblock Inequalities: theory of majorization and its applications.
\newblock 1979.

\bibitem[Nadeb and Torabi(2022)]{nadeb2022new}
Hossein Nadeb and Hamzeh Torabi.
\newblock New results on stochastic comparisons of finite mixtures for some
  families of distributions.
\newblock \emph{Communications in Statistics-Theory and Methods}, 51\penalty0
  (10):\penalty0 3104--3119, 2022.

\bibitem[Navarro and del {\'A}guila(2017)]{navarro2017stochastic}
Jorge Navarro and Yolanda del {\'A}guila.
\newblock Stochastic comparisons of distorted distributions, coherent systems
  and mixtures with ordered components.
\newblock \emph{Metrika}, 80:\penalty0 627--648, 2017.

\bibitem[Navarro and Hernandez(2004)]{navarro2004obtain}
Jorge Navarro and Pedro~J Hernandez.
\newblock How to obtain bathtub-shaped failure rate models from normal
  mixtures.
\newblock \emph{Probability in the Engineering and Informational Sciences},
  18\penalty0 (4):\penalty0 511--531, 2004.

\bibitem[Panja et~al.(2022)Panja, Kundu, and Pradhan]{panja2022stochastic}
Arindam Panja, Pradip Kundu, and Biswabrata Pradhan.
\newblock On stochastic comparisons of finite mixture models.
\newblock \emph{Stochastic Models}, 38\penalty0 (2):\penalty0 190--213, 2022.

\bibitem[Sattari et~al.(2021)Sattari, Barmalzan, and
  Balakrishnan]{sattari2021stochastic}
Mostafa Sattari, Ghobad Barmalzan, and Narayanaswamy Balakrishnan.
\newblock Stochastic comparisons of finite mixture models with generalized
  lehmann distributed components.
\newblock \emph{Communications in Statistics-Theory and Methods}, 51\penalty0
  (22):\penalty0 7767--7782, 2021.

\bibitem[Saunders and Moran(1978)]{saunders1978quantiles}
IW~Saunders and PAP Moran.
\newblock On the quantiles of the gamma and f distributions.
\newblock \emph{Journal of Applied Probability}, 15\penalty0 (2):\penalty0
  426--432, 1978.

\bibitem[Shaked and Shanthikumar(2007)]{shaked2007stochastic}
Moshe Shaked and J~George Shanthikumar.
\newblock \emph{Stochastic orders}.
\newblock Springer, 2007.

\bibitem[Shojaee and Babanezhad(2023)]{shojaee2023some}
Omid Shojaee and Manoochehr Babanezhad.
\newblock On some stochastic comparisons of arithmetic and geometric mixture
  models.
\newblock \emph{Metrika}, 86\penalty0 (5):\penalty0 499--515, 2023.

\bibitem[Shojaee et~al.(2021)Shojaee, Asadi, and Finkelstein]{shojaee2021some}
Omid Shojaee, Majid Asadi, and Maxim Finkelstein.
\newblock On some properties of $\alpha$-mixtures.
\newblock \emph{Metrika}, 84\penalty0 (8):\penalty0 1213--1240, 2021.

\bibitem[Shojaee et~al.(2022)Shojaee, Asadi, and
  Finkelstein]{shojaee2022stochastic}
Omid Shojaee, Majid Asadi, and Maxim Finkelstein.
\newblock Stochastic properties of generalized finite $\alpha$-mixtures.
\newblock \emph{Probability in the Engineering and Informational Sciences},
  36\penalty0 (4):\penalty0 1055--1079, 2022.

\bibitem[Yan and Niu(2023)]{yan2023stochastic}
Rongfang Yan and Jiale Niu.
\newblock Stochastic comparisons of second-order statistics from dependent and
  heterogenous modified proportional hazard rate observations.
\newblock \emph{Statistics}, 57\penalty0 (2):\penalty0 328--353, 2023.

\end{thebibliography}
\end{spacing}

\end{document}